\documentclass[twoside,12pt,leqno]{article}
\usepackage{amsmath, amsthm, amsfonts, amssymb, color}
\usepackage{mathrsfs}
\usepackage{fancyhdr}

\newcommand{\be}{\begin{eqnarray}}
\newcommand{\ee}{\end{eqnarray}}
\newcommand{\ce}{\begin{eqnarray*}}
\newcommand{\de}{\end{eqnarray*}}

\newtheorem{thm}{Theorem}[section]

\newtheorem{lem}[thm]{Lemma}

\theoremstyle{definition}
\newtheorem{defn}{Definition}[section]

\definecolor{wco}{rgb}{0.5,0.2,0.3}

\numberwithin{equation}{section} \theoremstyle{remark}
\newtheorem{rem}{Remark}[section]

\def\<{\langle} \def\>{\rangle}  
    
\def\d{\text{\rm{d}}}   
\def\E{\mathbb E}  
\def\beg{\begin} \def\beq{\begin{equation}}

 \def\P{\mathbb P} 
 
 \def\ee{\varepsilon}

\def\[{{\Big[}}
\def\]{{\Big]}}

\def\({{\Big(}}
\def\){{\Big)}}

\allowdisplaybreaks[1]

\begin{document}


\bigbreak

\title{ {\bf Existence and uniqueness of solutions to stochastic functional differential equations in infinite dimensions
\thanks{Research supported in part  by NSFC (No.11301026) and DFG through IRTG 1132 and CRC 701}\\} }
\author{{\bf Michael R\"{o}ckner}$^{\mbox{c},}$, {\bf Rongchan Zhu}$^{\mbox{a},}$\thanks{
 Corresponding author}, {\bf Xiangchan Zhu}$^{\mbox{b},}$
\date {}
\thanks{E-mail address:  roeckner@math.uni-bielefeld.de(M. R\"{o}ckner), zhurongchan@126.com(R. C. Zhu), zhuxiangchan@126.com(X. C. Zhu)}\\ \\
$^{\mbox{a}}$Department of Mathematics, Beijing Institute of Technology, Beijing 100081,
 China,\\
$^{\mbox{b}}$School of Science, Beijing Jiaotong University, Beijing 100044, China\\
$^{\mbox{c}}$ Department of Mathematics, University of Bielefeld, D-33615 Bielefeld, Germany,}

\maketitle

\begin{abstract}
\vskip 0.1cm \noindent In this paper, we present a general framework for solving  stochastic functional differential equations in infinite dimensions in the sense of martingale solutions, which can be applied to a large class of SPDE with finite delays, e.g. $d$-dimensional stochastic fractional Navier-Stokes equations with delays, $d$-dimensional stochastic reaction-diffusion equations with delays, $d$-dimensional stochastic porous media equations with delays. Moreover, under local monotonicity conditions for the nonlinear term  we obtain the existence and uniqueness of strong solutions to SPDE with delays.

\vspace{1mm}
\noindent{\footnotesize{\bf 2000 Mathematics Subject Classification AMS}:\hspace{2mm} 60H15, 60H30, 35R60}
 \vspace{2mm}

\noindent{\footnotesize{\bf Keywords}:\hspace{2mm}  stochastic functional equation, stochastic partial differential equations with delay, martingale problem, local monotonicity}

\end{abstract}
\vspace{1mm}

\section{Introduction}

Recently  stochastic partial differential equations (SPDE) with delays have been paid a lot of attention (see e.g [BBT14], [BDT12], [CR04] [TLT02], [J10] and the references therein). There is a large amount of literature on the mathematical theory and
on applications of stochastic functional (or delay) differential equations (see e.g. [M84], [RS08], [HMS09] and the reference therein).
When one wants to model some evolution phenomena arising in physics,
biology, engineering, etc., some hereditary characteristics such as after-effect,
time-lag or time-delay can appear in the variables. Typical examples arise in the mathematical modelling  of materials with termal memory, biochemical reactions,
population models, etc. (see, for instance, [R96], [W96] and references cited
therein). So, one is naturally lead to use  functional differential equations which take into account the
history of the system. However, in most cases,  randomness affects the model
, so that the system should be modelled by a stochastic
 functional equation.

There are a lot of studies about  existence and uniqueness of probabilistically strong solutions for various classes of
nonlinear SPDEs with time delays in a variational framework ( see e.g. [CGR02], [CLT00], [J10]): Let $$V\subset H_0\cong H_0^*\subset V^*$$
 be a Gelfand triple and
  \begin{equation}\begin{split}
&\d x(t)=[A_1(t,x(t))+A_2(t,x_t)] \d t+B(t,x_t) \d W(t), t\in[0,\infty),\\
 & x(s)=\psi(s), s\in[-h,0],
 \end{split}
\end{equation}
  where $W$ is a cylindrical Wiener process on a separable Hilbert space $U$,  $x_t(s):=x(t+s)$ and $A_1:\mathbb{R}^+\times V\to V^*$ is $\mathcal{B}(\mathbb{R}^+)\otimes\mathcal{B}(V)$-measurable,  $A_2:\mathbb{R}^+\times C([-h,0];H_0)\to V^*$ is $\mathcal{B}(\mathbb{R}^+)\otimes\mathcal{B}(C([-h,0];H_0))$-measurable, $B:
\mathbb{R}^+\times C([-h,0];H_0)\to L_{2}(U;H_0)$ is $\mathcal{B}(\mathbb{R}^+)\otimes\mathcal{B}(C([-h,0];H_0))$-measurable.
 In the above references the authors suppose that $A_1$, i.e. the nonlinear term without delay, satisfies monotonicity conditions, whereas $A_2$, i.e.  the nonlinear term with delay, satisfies Lipschitz conditions with respect to a suitable norm.
If we consider  more general equations with nonlinear terms not satisfying monotonicity conditions such as Navier-Stokes equations and more interesting delay terms like $A_2(t,x)=\int_{-h}^0 f(x(t+r))dr\cdot \nabla x(t-r_1(t))+\nabla b(x(t-r_2(t)))$, we cannot apply the above result. Here $r_1,r_2:\mathbb{R}\rightarrow[0,h]$ and $f$ is a bounded Lipschitz continuous function on $\mathbb{R}$, $b$ is a linear growth and continuous function on $\mathbb{R}$

To obtain the existence and uniqueness of solutions for the equations containing more general nonlinear terms and more interesting delays, we tried to use the monotonicity trick and assumed that the nonlinear terms satisfy local monotonicity conditions as Wei Liu and the first named author did in [LR10]. However, when we apply this method, the local monotonicity condition should be:
 There exist  locally bounded measurable functions $\rho: V\rightarrow[0,+\infty)$ and $\rho_1: C([-h,0];H_0)\rightarrow[0,+\infty)$ such that
\begin{equation}\aligned&{ }_{V^*}\!\langle A_1(t,\xi(t))-A_1(t,\eta(t))+ A_2(t,\xi_t)-A_2(t,\eta_t),\xi(t)-\eta(t)\rangle_V+\|B(t,\xi_t)-B(t,\eta_t)\|_2^2
\\\leq& [\rho(\eta(t))+\rho_1(\eta_t)]\|\xi(t)-\eta(t)\|_{{H_0}}^2, \endaligned\end{equation}
for $\xi,\eta\in C([-h,\infty),H_0)\cap L^p_{\rm{loc}}([-h,\infty);V)$ with $p\geq 1$. Here the middle norm is the $H_0$-norm, not the norm for the paths like $C([-h,0],H_0)$-norm, which is not natural and cannot even cover the Lipschitz case mentioned above.

Instead, in this paper we take a different approach. First, we provide  a general framework to prove  existence  of solutions (see Theorems 2.1 and 2.2) under very weak assumptions (see (H1), (H2), (H3)) in the sense of D. W. Stroock and S.R.S. Varadhan's martingale problem (see [SV79]), which can be applied to a large class of SPDEs with delays, such as  stochastic fractional Navier-Stokes equations with delay in any dimensions (see Section 5.1),  stochastic reaction-diffusion equations with delay (see Section 5.2) and stochastic porous media equations with delay in any dimension (see Section 5.3). We also emphasize that for our existence results, we only assume continuity, coercivity and growth conditions written in terms of integrals over time, which enables us to cover a large number of equations with interesting delays (see Remark 2.2 (ii)). Second, under local monotonicity conditions for the nonlinear terms  we obtain  pathwise uniqueness for SPDEs with delays, which implies existence and uniqueness of (probabilistically) strong solution by the Yamada-Watanabe Theorem.  Here we also emphasize that our local monotonicity condition (H4) can cover the examples we mentioned above. It is clearly weaker than the Lipschitz condition and (1.2), since instead of the middle norm $H_0$ we use  a norm for the paths and the local term in the local monotonicity condition may depend on the paths of two solutions.

This paper is organized as follows: In Section 2, we introduce the general framework and state the main results on  existence of  martingale solutions. In Section 3, we give the proofs of the main results from Section 2. In Section 4 we concentrate on the uniqueness of  probabilistically strong solutions. In Section 5, we apply the main results from Sections 2 and 4 to stochastic fractional Navier-Stokes equations with delays in any dimensions (see Section 5.1),  to stochastic reaction-diffusion equations with delays (see Section 5.2) and to stochastic porous media equations with delays in any dimensions (see Section 5.3).

\section{Existence of Martingale solutions}

Let
  $(\mathbb{H}, \<\cdot,\cdot\>_\mathbb{H})$ be a separable
Hilbert space and $\mathbb{X}, V,  \mathbb{Y}$ be three separable and reflexive  Banach spaces with norms $\|\cdot\|_\mathbb{X}$, $\|\cdot\|_V$ and $\|\cdot\|_\mathbb{Y}$ such that
 $$\mathbb{X}\subset V\subset\mathbb{Y}\subset\mathbb{H}$$
 continuously and densely. By Kuratowski's theorem we have that $\mathbb{X}\in\mathcal{B}(\mathbb{V})$, $V\in\mathcal{B}(\mathbb{Y})$, $\mathbb{Y}\in\mathcal{B}(\mathbb{H})$ and $\mathcal{B}(\mathbb{Y})=\mathcal{B}(\mathbb{H})\cap \mathbb{Y}, \mathcal{B}(\mathbb{X})=\mathcal{B}(V)\cap \mathbb{X}, \mathcal{B}(V)=\mathcal{B}(\mathbb{Y})\cap V$.

  If we identify the dual of $\mathbb{H}$ with itself, we get
 $$\mathbb{X}\subset\mathbb{H}\simeq\mathbb{H}^*\subset\mathbb{X}^*$$
 continuously and densely.

 If  ${ }_{\mathbb{X}^*}\<\cdot,\cdot\>_{\mathbb{X}}$ denotes  the dualization
between  $\mathbb{X}$ and its dual space $\mathbb{X}^*$,  then it follows that
$$ { }_{\mathbb{X}^*}\<u, v\>_{\mathbb{X}}=\<u, v\>_\mathbb{H}, \  v\in \mathbb{X} ,u\in \mathbb{H}^*.$$
Moreover, we assume that
 $\mathbb{X}$ is a Hilbert space and $\mathbb{X}\subset \mathbb{H}$ compactly.

 Let
$$\{e_1,e_2,\cdots \}\subset\mathbb{X}$$ be an orthonormal basis of $\mathbb{H}$ and
 let $\mathbb{H}_n:=span\{e_1,\cdots,e_n\}$ and $\mathcal{E}:=span\{e_1,e_2,\cdots\}$ is dense in $\mathbb{X}^*$. Let $\Pi_n:\mathbb{X}^*\rightarrow \mathbb{H}_n$ be defined by
$$ \Pi_ny:=\sum_{i=1}^n { }_{\mathbb{X}^*}\<y,e_i\>_{\mathbb{X}} e_i, \ y\in \mathbb{X}^*.  $$
Since $\mathbb{X}\subset \mathbb{H}$ is compact and hence so is $\mathbb{H}(\simeq\mathbb{H}^*)\subset\mathbb{X}^*$, we can choose $\{e_n:n\in\mathbb{N}\}\subset\mathbb{X}$ in such a way that
$$\|\Pi_ny\|_{\mathbb{X}*}\leq \|y\|_{\mathbb{X}*}, \quad\forall n\in\mathbb{N}, y\in\mathbb{X}^*,$$
(cf the proof of [AR89, Proposition 3.5]).

\begin{rem} $\mathbb{X}^*$ will be the space where the SPDE will initially hold. $\mathbb{H}$ will eventually be the state space for the solutions of the equation. $\mathbb{Y}$ and $V$ are spaces to identify integrability properties of the solutions. In applications we choose $\mathbb{Y}=\mathbb{H}$ (see examples in Sections 5.1 and 5.2) or $\mathbb{Y}=V$ (see the example in Section 5.3).   \end{rem}

Let $\{W_t\}_{t\geq0}$ be a cylindrical Wiener process on a
separable
 Hilbert space $U$ w.r.t a complete filtered probability space
$(\Omega,\mathcal{F}^0,(\mathcal{F}_t^0),\mathbb{P})$ and $(L_2(U;\mathbb{H}),$ $
 \|\cdot\|_{L_2(U,\mathbb{H})}) $  denotes
the space of all Hilbert-Schmidt operators from $U$ to $\mathbb{H}$.
Let $$\mathfrak{C}:=C([-h,0];\mathbb{X}^*)\cap L^\infty(-h,0;\mathbb{H}),$$  with norm
$$\|\psi\|_\mathfrak{C}:=\sup_{-h\leq s\leq0}\|\psi(s)\|_\mathbb{H},\quad \psi\in \mathfrak{C},$$
and $$\mathfrak{C}_\infty:=C([-h,\infty);\mathbb{X}^*)\cap L_{\rm{loc}}^\infty([-h,\infty);\mathbb{H}),$$
equipped with the corresponding metric,
$$L_\mathbb{Y}^p:=L^p([-h,0];\mathbb{Y}),\quad L_V^p:=L^p([-h,0];V)\quad \textrm{ and }\quad L_\mathbb{H}^p:=L^p([-h,0];\mathbb{H}),\quad  p>1.$$
Since $\mathbb{H}\subset \mathbb{X}^*$ is compact, it is easy to see that
$$\mathfrak{C}=C_{bw}([-h,0];\mathbb{H}),$$
where the latter denotes the set of all norm-bounded weakly continuous functions from $[-h,0]$ to $\mathbb{H}$. Then obviously $\mathfrak{C}$ with the above norm is Polish.

Given a path $y \in C([-h,+\infty);\mathbb{X}^*)$ and $t\geq0$, we associate with it a  path  $y_t\in C([-h,0];\mathbb{X}^*)$ by setting $y_t(s)=y(t+s),s\in[-h,0]$.

We
consider the following stochastic evolution equation
 \begin{equation}\begin{split}\label{SEE}
&\d x(t)=[A_1(t,x(t))+A_2(t,x_t)] \d t+B(t,x_t) \d W(t), t\in[0,\infty),\\
 & x(s)=\psi(s), s\in[-h,0],
 \end{split}
\end{equation}
 where  $A_1:\mathbb{R}^+\times V\to \mathbb{X}^*$ is $\mathcal{B}(\mathbb{R}^+)\otimes\mathcal{B}(V)$-measurable,  $A_2:\mathbb{R}^+\times (L^p_V\cap \mathfrak{C})\to \mathbb{X}^*$ is $\mathcal{B}(\mathbb{R}^+)\otimes\mathcal{B}(L^p_V\cap\mathfrak{C})$-measurable, $B:
\mathbb{R}^+\times (L^p_V\cap\mathfrak{C})\to L_{2}(U;\mathbb{H})$ is $\mathcal{B}(\mathbb{R}^+)\otimes\mathcal{B}(L^p_V\cap\mathfrak{C})$-measurable for some $p\geq2$.   Now we fix  $p$ and suppose that
for every $ u  \in \mathbb{S}_1:=L^p_{\rm{loc}}([-h,\infty);V)\cap \mathfrak{C}_\infty$,
$$A_1(\cdot, u (\cdot))\in L^1_{\rm{loc}}([0,\infty);\mathbb{X}^*), A_2(\cdot, u_\cdot)\in L^1_{\rm{loc}}([0,\infty);\mathbb{X}^*).$$

Let us now state the precise conditions on the coefficients of equation
(\ref{SEE}):

First we introduce the following function class $\mathfrak{U}^q, q\geq1$: A lower semi-continuous function
$\mathcal{N}:\mathbb{Y}\rightarrow[0,\infty]$ belongs to $\mathfrak{U}^q$ if $\mathcal{N}(y)=0$ implies $y=0$, and
$$\mathcal{N}(cy)\leq c^q\mathcal{N}(y), \quad \forall c\geq0, y\in\mathbb{Y},$$
and
$$\{y\in\mathbb{Y}:\mathcal{N}(y)\leq 1\} \textrm{ is  compact  in } \mathbb{Y}.$$

Then we can extend $\mathcal{N}$ to a $\mathcal{B}(\mathbb{X}^*)/\mathcal{B}([0,\infty])$-measurable function on $\mathbb{X}^*$ by setting $\mathcal{N}(x)=\infty, x\in \mathbb{X}^*\backslash\mathbb{Y}$.

Suppose that there exists $\mathcal{N}_1\in \mathfrak{U}^p$ for  $p\geq2$ as above such that
\begin{equation}\label{bd8}\|x\|_V^p\lesssim\mathcal{N}_1(x), \quad \forall x\in\mathbb{Y},\end{equation}
where $\|x\|_V:=+\infty$ if $x\in\mathbb{Y}\backslash V$, and  the following holds: $\|\Pi_nA_1(\cdot,\cdot)\|_{\mathbb{X}^*}$ is  locally bounded on $\mathbb{R}^+\times \mathbb{H}_n$,
$\|\Pi_nA_2(\cdot,\cdot)\|_{\mathbb{X}^*}, \|\Pi_nB(\cdot,\cdot)\|_{L_2(U,\mathbb{H})} $ are  bounded on balls in $\mathbb{R}^+\times C([-h,0);\mathbb{H}_n)$ and
 for every $t>0$, $\Pi_n A_1(t,\cdot)$ is continuous on $\mathbb{H}_n$, $\Pi_n A_2(t,\cdot), \Pi_n B(t,\cdot)$ are continuous on $C([-h,0];\mathbb{H}_n)$ and $\mathcal{N}_1(x)\leq C_n \|x\|_{\mathbb{H}_n}^p$ for $x\in\mathbb{H}_n$.

(H1) (Weakened demicontinuity) If $\{u^n,u, n\in\mathbb{N}\}\subset \mathfrak{C}_\infty$ with  $ u^n$ converging to $ u $ in $\mathbb{S}_0:=L^p_{\rm{loc}}([-h,\infty);\mathbb{Y})\cap C([-h,\infty);\mathbb{X}^*) $ and weakly in $L^p_{\rm{loc}}([-h,\infty);V)$, then for every $R>0$ and  cut-off function $\chi_R\in C_0^\infty(\mathbb{R})$ satisfying
$$\chi_R(x)=\left\{\begin{array}{ll}1&\ \ \ \ \textrm{ for } |x|\leq R,\\0&\ \ \ \ \textrm{ for } |x|>2R,\end{array}\right.$$ we have for all $v\in\mathbb{X}, t\geq0$,
\begin{equation}\label{cd1}\lim_{n\rightarrow\infty}\int_0^t{ }_{\mathbb{X}^*}\!\langle A_1(s, u^n(s)),v\rangle_{\mathbb{X}}\textrm{ }\chi_R(\| u ^n(s)\|_\mathbb{H})\d s=\int_0^t{ }_{\mathbb{X}^*}\!\langle A_1(s, u (s)),v\rangle_{\mathbb{X}}\textrm{ }\chi_R(\| u (s)\|_\mathbb{H})\d s,\end{equation}
 \begin{equation}\label{cd2}\lim_{n\rightarrow\infty}\int_0^t{ }_{\mathbb{X}^*}\!\langle A_2(s, u _s^n),v\rangle_{\mathbb{X}}\d s=\int_0^t{ }_{\mathbb{X}^*}\!\langle A_2(s, u _s),v\rangle_{\mathbb{X}}\d s,\end{equation}
and
$$\lim_{n\rightarrow\infty}\int_0^t\|B^*(s, u ^n_s)(v)-B^*(s, u _s)(v)\|_U\d s=0,$$
where $B^*$ denotes the adjoint operator of $B$.

(H2) (Coercivity) There exist  locally bounded measurable functions $\lambda_1, c_1:\mathbb{R}\mapsto\mathbb{R}^+$ such that for $  u \in C([-h,\infty);\mathbb{X})$, $t\geq0$ with $\int_0^t\mathcal{N}_1(u(s))\d s<\infty$
$$\aligned&\int_0^t{ }_{\mathbb{X}^*}\!\langle A_1(s, u (s))+A_2(s, u _s), u (s)\rangle_{\mathbb{X}}\d s\\\leq& -\frac{1}{2}\int_0^t\mathcal{N}_1( u (s))\d s+c_1(t)\int_{-h}^0\mathcal{N}_1( u (s))\d s+\lambda_1(t)\int_0^t(1+\| u _s\|_\mathfrak{C}^2)\d s.\endaligned$$

(H3) (Growth condition)  There exist locally bounded measurable functions $\lambda_2, \lambda_3, \lambda_4:\mathbb{R}\mapsto\mathbb{R}^+$ and $\gamma'\geq \gamma>1$ such that for all  $ u \in \mathbb{S}_1$, $t\geq0$
$$\int_0^t\|A_1(s, u (s))\|_{\mathbb{X}^*}^{\gamma}\d s\leq \lambda_2(t)[\int_0^t\mathcal{N}_1( u (s))\d s]^{\gamma'}+\lambda_3(t)(1+\sup_{s\in[0,t]}\| u (s)\|^{\gamma'}_\mathbb{H}),$$
$$\int_0^t\|A_2(s, u _s)\|_{\mathbb{X}^*}^{\gamma}\d s\leq \lambda_2(t)[\int_{-h}^t\mathcal{N}_1( u (s))\d s]^{\gamma'}+\lambda_3(t)(1+\sup_{s\in[-h,t]}\| u (s)\|^{\gamma'}_\mathbb{H}),$$
$$\|B(t, u _t)\|_{L_2(U;\mathbb{H})}^2\leq \lambda_4(t)(1+\| u _t\|^2_\mathfrak{C}).$$

\begin{rem} (i) If  $y^n$ strongly converges to $y$ in $\mathbb{Y}$, for every $v\in\mathbb{X}, s\in\mathbb{R}^+$,
$$\lim_{n\rightarrow\infty}{ }_{\mathbb{X}^*}\!\langle A_1(s,y^n),v\rangle_{\mathbb{X}}={ }_{\mathbb{X}^*}\!\langle A_1(s,y),v\rangle_{\mathbb{X}},$$
and $$\mathcal{N}_1(\cdot)\asymp\|\cdot\|_V^p,$$
 (\ref{cd1}) in (H1) holds by (H3) and the dominated convergence theorem. In this case
 (\ref{cd2}) is equivalent to the following demicontinuity: if
$u^n$ converge to $u$ in $\mathbb{S}_0$ and weakly in $L^p_{\rm{loc}}([-h,\infty);V)$,
$A_2(\cdot,u_\cdot^n)$ converge to $A_2(\cdot,u_\cdot)$ weakly in $L^1_{\rm{loc}}([0,\infty);\mathbb{X}^*)$.

(ii) Since we consider stochastic  differential equations with delay, the coercivity conditions and growth conditions we assume are written in terms of integrals over time, more precisely in the dualization between $L^1_{\rm{loc}}([-h,\infty);\mathbb{X}^*)$ and $L^\infty_{\rm{loc}}([-h,\infty);\mathbb{X})$, which is of course weaker than the general coercivity conditions and growth conditions without time integrals. For (H2) and (H3) it is very natural to use $\int_{-h}^0\mathcal{N}_1( u (s))\d s$ to control the growth of the nonlinear term, which also enables us to cover more interesting nonlinear terms with delay (see Section 5).
\end{rem}

Let $\Omega_0:=C([-h,\infty);\mathbb{X}^*)$  with the metric
$$\rho(x,y):=\sum_{m=1}^\infty\frac{1}{2^m}(\sup_{s\in[-h,m]}\|x(s)-y(s)\|_{\mathbb{X}^*}\wedge1).$$
Define the canonical process $x:\Omega_0\rightarrow \mathbb{X}^*$ as $x(t,\omega):=\omega(t)$. For $t\geq0$ define the $\sigma$-algebra by $$\mathcal{F}_t:=\sigma\{x(r):r\leq t\}.$$
and $-h\leq t\leq0$, $\mathcal{F}_t:=\mathcal{F}_0$, $\mathcal{F}:=\vee_{t\geq0}\mathcal{F}_t$.

\beg{defn} (Martingale Solution) Given an initial value $\psi\in \mathfrak{C}$, a probability measure $P\in\mathcal{P}(\Omega_0)$ (=all probability measure on $\Omega_0$),  is called a martingale solution of Eq. (1.1) if

(M1) $P(x(t)=\psi(t), -h\leq t\leq 0)=1$ and for every $n\in\mathbb{N}$
$$P\left\{x\in\Omega_0:\int_0^n\|A_1(s,x(s))+A_2(s,x_s)\|_{\mathbb{X}^*} \d s+\int_0^n\|B(s,x_s)\|_{L_2(U;\mathbb{H})}^2\d s<+\infty\right\}=1;$$
(M2) for every $l\in\mathcal{E}$ the process
$$M_l(t,x):={ }_{\mathbb{X}^*}\<x(t),l\>_{\mathbb{X}}-\int_0^t{ }_{\mathbb{X}^*}\<A_1(s,x(s))+A_2(s,x_s),l\>_{\mathbb{X}}ds, t\geq0,$$
is a continuous square integrable $(\mathcal{F}_t)$-martingale with respect to $P$, whose quadratic variation process is given by
$$\langle M_l\rangle(t,x):=\int_0^t\|B^*(s,x_s)(l)\|_U^2\d s,t\geq0.$$
\end{defn}

\begin{rem} By Kuratowski's theorem $L_V^p\in\mathcal{B}(\mathfrak{C})$ and $\mathcal{B}(L_V^p\cap \mathfrak{C})=\mathcal{B}(\mathfrak{C})\cap L_V^p$. Furthermore,  $A_2:\mathbb{R}^+\times (L^p_V\cap \mathfrak{C})\to \mathbb{X}^*$ is $\mathcal{B}(\mathbb{R}^+)\otimes\mathcal{B}(L^p_V\cap\mathfrak{C})$-measurable and the map $x\rightarrow x_s$ is $\mathcal{F}_s$-measurable, which imply that for $s\geq0$, $A_2(s,x_s)$ is $(\mathcal{F}_s)$-adapted.
\end{rem}

Now we can state the first main result.

\beg{thm}\label{T1}
 Suppose $(H1)-(H3)$ hold for some $p\geq 2$.
    Then for every $\psi\in \mathfrak{C}$ satisfying
    \begin{equation}\label{ini}\mathcal{N}(\psi):=\sup_n\int_{-h}^0\mathcal{N}_1(\Pi_n\psi(t))\d t+\int_{-h}^0\mathcal{N}_1(\psi(t))\d t<\infty,\end{equation}
    $(\ref{SEE})$
    has a martingale solution $P$ such that for every $q\in\mathbb{N}$, $T>0$
$$\E^P\left(\sup_{t\in[-h,T]}\|x(t)\|_\mathbb{H}^{2q}+[\int_0^T\mathcal{N}_1(x(t))\d t]^q \right)   < \infty.$$

\end{thm}

\begin{rem} (i) Condition (\ref{ini}) ensures the crucial  energy estimate in Lemma \ref{L2} below by the coercivity condition (H2). If $\mathcal{N}_1(\Pi_n\psi)\leq \mathcal{N}_1(\psi)$ for $\psi\in V$, which  holds in the examples in Sections 5.1, 5.2,  then condition (\ref{ini}) is satisfied if $\int_{-h}^0\mathcal{N}_1(\psi(t))\d t<\infty$. In general,  condition (\ref{ini})  holds if $\psi$ is smooth enough.

(ii) Condition (\ref{ini}) can  be dropped if $c_1=0$ in (H2) and  (\ref{cd2}) in (H1) holds if $ u^n$ converge to $ u $ in $L^p_{\rm{loc}}([-h,\infty);\mathbb{H})\cap C([-h,\infty);\mathbb{X}^*)$.
\end{rem}

\begin{defn} We say that there exists a weak solution of  equation (2.1) if given  every $\psi\in \mathfrak{C}\cap L^p_V$, there exists a stochastic basis $(\Omega,\mathcal{F},\{\mathcal{F}_t\}_{t\in [-h,\infty)},P)$, a cylindrical Wiener process $W$ on the space $U$ and a progressively measurable process $X:[-h,\infty)\times \Omega\rightarrow \mathbb{H}$, such that for  $P$-a.e. $\omega\in \varOmega$,
$$X(\cdot,\omega)\in \mathbb{S}_1,$$  and such that $P$-a.s.
$$X(t)=\psi(t), -h\leq t\leq 0,$$
and
$${ }_{\mathbb{X}^*}\langle X(t),l\rangle_{\mathbb{X}}-\int_0^t{ }_{\mathbb{X}^*}\langle A_1(s,X(s))+A_2(s,X_s),l\rangle_{\mathbb{X}} \d s
=\langle X(0),l\rangle_\mathbb{H}+\langle \int_0^tB(s,X_s)\d W(s),l\rangle_\mathbb{H},$$  for  $t\in [0,\infty)$ and all $l\in \mathbb{X}$.
\end{defn}

Moreover, by [On05, Theorem 2] and Theorem \ref{T1} we obtain the following results.

\beg{thm}\label{T2}
 Suppose $(H1)-(H3)$ hold for some $p\geq2$.
    Then for every $\psi\in   \mathfrak{C}$ with $\mathcal{N}(\psi)<\infty$,
    $(\ref{SEE})$
    has a weak solution $X:[-h,\infty)\times \Omega\rightarrow \mathbb{H}$ such that for every $q\in\mathbb{N}$, $T>0$
$$\E^P\left(\sup_{t\in[-h,T]}\|X(t)\|_\mathbb{H}^{2q}+[\int_{-h}^T\mathcal{N}_1(X(t))dt]^q \right)   < \infty.$$

\end{thm}

\section{Proof of Theorem \ref{T1}}

\subsection{Finite dimensional case}

Now we prove Theorem \ref{T1} for $U=\mathbb{H}=\mathbb{X}=\mathbb{R}^d$. In this case $$\Omega_0=C([-h,\infty);\mathbb{R}^d),$$
$$\mathfrak{C}:=C([-h,0];\mathbb{R}^d),$$
 $$A(t, u ):=A_1(t, u (0))+A_2(t, u ), \quad \forall  u \in \mathfrak{C} .$$
Define for $n\in\mathbb{N}$, $ x \in C([-h,\infty);\mathbb{R}^d)$
$$A^n(t, x ):=1_{\{t\leq n\}}\chi_n(\sup_{-h\leq s\leq t}\|x(s)\|_{\mathbb{R}^d})A(t, x_t), \quad B^n(t, x ):=1_{\{t\leq n\}}\chi_n(\sup_{-h\leq s\leq t}\|x(s)\|_{\mathbb{R}^d})B(t, x_t),$$
where $0\leq \chi_n\in C(\mathbb{R}^+;\mathbb{R})$ is a decreasing cutoff function with
$$\chi_n( r )=\left\{\begin{array}{ll}1&\ \ \ \ \textrm{ for } r\leq n,\\0&\ \ \ \ \textrm{ for } r>2n,\end{array}\right.$$
 Then for $n\in\mathbb{N}$, $A^n(\cdot,x(\cdot))$ and $B^n(\cdot,x(\cdot))$ are bounded and progressively measurable and for each $t>0$  $A^n(t,x(\cdot)), B^n(t,x(\cdot))$ are  continuous on $\Omega_0$ and $\mathcal{F}_t$-measurable. By a similar argument as in the proof of [SV79, Theorem 6.1.6] we obtain that there exists a probability measure $P_n\in\mathcal{P}(\Omega_0)$ such that $P_n(x(t)=\psi(t), -h\leq t\leq 0)=1$ and
$$M_n(t,x):=x(t)-x(0)-\int_0^tA^n(s,x)\d s, x\in \Omega_0,$$
is a continuous square integrable $(\mathcal{F}_t)$-martingale with square variation process
$$\langle M_n\rangle(t,x)=\int_0^t(B^n)^*(s,x)B^n(s,x)ds,t\geq0.$$
Although in [SV79, Theorem 6.1.6] it requires that the initial value takes values in $\mathbb{R}^d$, we can still use a similar argument as in the proof to obtain the above results even if the initial value belongs to $\mathfrak{C}$.

 By   It\^{o}'s formula and (H2), (H3) we have
\beq\aligned
&\|x(t)\|_\mathbb{H}^2\\=&\|x(0)\|_\mathbb{H}^2+\int_0^{t}\left( 2 \<A^n(s,x),x(s)\>+\|B^n(s,x)\|_{L_2(\mathbb{R}^d;\mathbb{R}^d)}^2\right)
\d s
+2\int_0^{t} x(s)\d M_n(s)\\=&\|x(0)\|_\mathbb{H}^2+\int_0^{t}\left(  1_{\{s\leq n\}}\chi_n(\sup_{-h\leq r\leq s}\|x(r)\|_{\mathbb{R}^d})[2\<A(s,x_s),x(s)\>+\|B^n(s,x_s)\|_{L_2(\mathbb{R}^d;\mathbb{R}^d)}^2]\right)
\d s
\\&+2\int_0^{t} x(s)\d M_n(s)\\=&\|x(0)\|_\mathbb{H}^2+ \chi_n(\sup_{-h\leq r\leq 0}\|x(r)\|_{\mathbb{R}^d})\int_0^{c}\left( [2\<A(s,x_s),x(s)\>+\|B(s,x_s)\|_{L_2(\mathbb{R}^d;\mathbb{R}^d)}^2]\right)
\d s
\\&+2\int_0^{t} x(s)\d M_n(s)\\
\le & \|\psi(0)\|_\mathbb{H}^2 + C \int_0^{t} \left( \|x_s\|_\mathfrak{C}^{2}
+ 1\right)\d s
+2\int_0^{t} x(s)\d M_n(s)+C\mathcal{N}(\psi),
\endaligned
\end{equation}
where $0\leq c\leq t\wedge n$, $C$ is a generic constant (independent of $n$) and we used mean value theorem for integrations in the third equality.
By the BDG(=Burkholder-Davis-Gundy) inequality, (H3) and Young's inequality we have
$$\aligned&\E^{{P}_n}(\sup_{r\in[0,t]}\|x(r)\|_\mathbb{H}^2)\\\leq& \|\psi(0)\|_\mathbb{H}^2 + C \E^{{P}_n}\int_0^t \left( \|x_s\|_\mathfrak{C}^{2}
+ 1\right)\d s
+2\E^{{P}_n} (\sup_{r\in[0,t]}|\int_0^r x(s)\d M_n(s)|)\\&+C\mathcal{N}(\psi)\\\leq&\|\psi(0)\|_\mathbb{H}^2 + C \E^{{P}_n}\int_0^t \left( \|x_s\|_\mathfrak{C}^{2}
+ 1\right)\d s+C\mathcal{N}(\psi)
\\&+C\E^{{P}_n} (\sup_{r\in[0,t]}\|x(r)\|_\mathbb{H}^2\int_0^t\|B(s,x_s)\|^2_{L_2(\mathbb{R}^d;\mathbb{R}^d)} \d s)^{1/2}\\\leq &\|\psi(0)\|_{\mathbb{H}}^2 + C \E^{{P}_n}\int_0^t \left( \|x_s\|_\mathfrak{C}^{2}
+ 1\right)\d s+C\mathcal{N}(\psi)
+\frac{1}{2}\E^{{P}_n}\sup_{r\in[0,t]}\|x(r)\|_\mathbb{H}^2.\endaligned$$
Hence by Gronwall's inequality we obtain for every $T>0$
$$\E^{{P}_n}(\sup_{r\in[-h,T]}\|x(r)\|_\mathbb{H}^2)\leq C_T(1+\sup_{t\in[-h,0]}\|\psi(t)\|_\mathbb{H}^2+\mathcal{N}(\psi)).$$
Similar as in the proof of Lemma 3.1 below we obtain that for
$q\in\mathbb{N}$ there exists $C_T>0$ such
that for all $n\in \mathbb{N}$
 we have
\begin{equation}\label{bdf1}
\E^{{P}_n}\sup_{t\in[-h,T]}\|x(t)\|_\mathbb{H}^{2q}
 \le C_T\left(\|\psi\|_\mathfrak{C}^{2q}+ 1+\mathcal{N}(\psi)^q\right),
\end{equation}
which combining with our assumption for $\mathcal{N}_1$ in finite dimensional case also implies that
$$
\E^{{P}_n}\int_{-h}^T\mathcal{N}_1(x(t))\d t
 \le C_T\left(\|\psi\|_\mathfrak{C}^{2p}+ 1+\mathcal{N}(\psi)^p\right),
$$

Moreover, as in infinite dimensional case (cf. Section 3.2 below) we have for some $\beta\in (0,1)$
$$\sup_n\E^{{P}_n}\left(\sup_{s\neq t\in[0,T]}\frac{|x(t)-x(s)|}{|t-s|^\beta}\right)\leq C,$$
which implies that for every $T>0$
 $({P}_n)_{n\in\mathbb{N}}$ is tight on $C([-h,+\infty);\mathbb{R}^d)$. Selecting a subsequence if necessary, we may assume that ${P}_n$ weakly converges to some probability measure $P$ on $C([-h,+\infty);\mathbb{R}^d)$. By Skorohod's representation theorem, there exists a probability space $(\tilde{\Omega},\tilde{\mathcal{F}},\tilde{P})$ and $C([-h,+\infty);\mathbb{R}^d)$-valued random variables $\tilde{x}^{n}$ and $\tilde{x}$ such that

(i) $\tilde{x}^{n}$ has the law ${P}_n$ for each $n\in\mathbb{N}$;

(ii) $\tilde{x}^{n}\rightarrow\tilde{x}$ in $C([-h,+\infty);\mathbb{R}^d)$, $\tilde{P}$-a.e., and $\tilde{x}$ has the law $P$.

 By similar arguments as in the infinite dimensional case, we have obtained  the existence of  martingale solutions in the finite dimensional case.

\subsection{Infinite dimensional case}

The first step of the proof  is mainly based on the Galerkin approximation.

Obviously, $\Pi_n|_\mathbb{H}$ is just the orthogonal projection onto $\mathbb{H}_n$ in $\mathbb{H}$ and we have
$$ { }_{\mathbb{X}^*}\<\Pi_nA_1(t,y), v\>_{\mathbb{X}}=\<\Pi_n A_1(t,y),v\>_\mathbb{H}={ }_{\mathbb{X}^*}\<A_1(t,y),v\>_{\mathbb{X}},  y\in  \mathbb{Y}, v\in \mathbb{H}_n, $$
and
$$ { }_{\mathbb{X}^*}\<\Pi_nA_2(t, u ), v\>_{\mathbb{X}}=\<\Pi_n A_2(t, u ),v\>_\mathbb{H}={ }_{\mathbb{X}^*}\<A_2(t, u ),v\>_{\mathbb{X}},   u \in \mathfrak{C}\cap L^p_\mathbb{Y}, v\in \mathbb{H}_n.  $$
Let $\{g_1,g_2,\cdots \}$ be an orthonormal basis of $U$ and
$$ W^{(n)}_t:=\sum_{i=1}^n\<W_t,g_i\>_U g_i=\tilde{\Pi}_n W_t, $$
where $\tilde{\Pi}_n$ is the orthogonal projection onto $span\{g_1,\cdots,g_n\}$ in $U$.

Now for each
finite $n\in \mathbb{N}$  we consider the following stochastic
equation on $\mathbb{H}_n$
\begin{equation}\begin{split}\label{approximation}
&\d x^{(n)}(t)=[\Pi_n A_1(t,x^{(n)}(t))+\Pi_n A_2(t,x^{(n)}_t)] \d t+ \Pi_n B(t,x_t^{(n)}) \d
W^{(n)}(t),\\
 & \ x^{(n)}(t)=\Pi_n \psi(t), t\in[-h,0].
\end{split}\end{equation}
Set $$\Omega^{(n)}:=C([-h,\infty);\mathbb{H}_n)$$
and
$$\mathcal{F}_t^{(n)}:=\mathcal{B}(C([-h,t];\mathbb{H}_n)), \mathcal{F}^{(n)}:=\vee_{t\geq0}\mathcal{F}_t^{(n)}.$$

By  the solvability of SDDE in the finite dimensional case (see Section 3.1) we know that
(\ref{approximation}) has a martingale solution, i.e. there exists a probability measure $P_n\in\mathcal{P}(\Omega^{(n)})$ such that (M1) and (M2) hold.

In order to construct the solution of (\ref{SEE}), we need some a
priori estimates for $x^{(n)}$.

\begin{lem}\label{L2}
 Under the assumptions in Theorem \ref{T1}, for every $q\in\mathbb{N}, T>0$ there exists $C_T>0$ such
that for all $n\in \mathbb{N}$
 we have
\begin{equation}\label{l2}
\E^{P_n}\sup_{t\in[-h,T]}\|x^{(n)}(t)\|_\mathbb{H}^{2q} +\E^{P_n}(\int_{-h}^T \mathcal{N}_1(x^{(n)}(t)) \d t)^q
 \le C_T\left(\|\psi\|_\mathfrak{C}^{2q}+ 1+\mathcal{N}(\psi)^q\right).
\end{equation}
\end{lem}

\begin{proof} First by (M2) the following equality holds in $\mathbb{H}_n$
$$x^{(n)}(t)=\Pi_nx(0)+\int_0^t[\Pi_nA_1(r,x^{(n)}(r))+\Pi_nA_2(r,x_r^{(n)})]\d r+M_n(t,x^{(n)}),t\in[0,T],$$
where $M_n(t,x^{(n)}),t\in[0,T],$ is a continuous $\mathbb{H}_n$-valued $(\mathcal{F}_t^{(n)})$-martingale with respect to $P_n$, whose covariation operator process in $\mathbb{H}_n$ is given by
$$\langle M_n\rangle(t,x^{(n)})=\int_0^t\Pi_nB(s,x^{(n)}_s)\tilde{\Pi}_n\tilde{\Pi}_n^*B^*(s,x^{(n)}_s)\Pi_n^*\d s, t\geq0.$$
 By   It\^{o}'s formula and (H2)  we have
\beq\begin{split}\label{Ito estimate 2}
\|x^{(n)}(t)\|_\mathbb{H}^{2}=&\|x^{(n)}(0)\|_\mathbb{H}^{2}+\int_0^t[2\<A_1(s,x^{(n)}(s))+A_2(s,x^{(n)}_s),x^{(n)}(s)\>_{\mathbb{H}_n}
\\
&+\|\Pi_nB(s,x_s^{(n)})\tilde{\Pi}_n\|_{L_2(U,\mathbb{H})}^2]
\d s
+M_n^{(1)}(t,x^{(n)})\\
\le & \|\psi(0)\|_\mathbb{H}^{2}+\int_0^t\left(-\mathcal{N}_1(x^{(n)}(s))+(2\lambda_1(t)+\lambda_4(t))(1+\|x^{(n)}_s\|_\mathfrak{C}^2)
\right)
\d s+2c_1(t)\mathcal{N}(\psi)\\
&+M_n^{(1)}(t,x^{(n)}),
 \  \ t\in[0,T],
\end{split}
\end{equation}
where $M_n^{(1)}(t,x^{(n)})$ is a continuous real-valued $(\mathcal{F}_t^{(n)})$-martingale with respect to $P_n$, whose quadratic variation process is given by
$$\langle M_n^{(1)}\rangle(t,x^{(n)}):=4\int_0^t\|
(\Pi_nB(s,x_s^{(n)})\tilde{\Pi}_n)^* x^{(n)}(s)\|_{U}^2\d s,t\geq0.$$
By (\ref{Ito estimate 2}) we obtain that for $t\in[0,T]$
$$\|x^{(n)}(t)\|_\mathbb{H}^{2}+\int_0^t\mathcal{N}_1(x^{(n)}(s))\d s
\le  \|\psi(0)\|_\mathbb{H}^{2}+(2\lambda_1(t)+\lambda_4(t))\int_0^t(1+\|x^{(n)}_s\|_\mathfrak{C}^2)
\d s+2c_1(t)\mathcal{N}(\psi)+M_n^{(1)}(t,x^{(n)}),$$
which implies that for every $q\in \mathbb{N}$
\begin{equation}\label{Ito}\|x^{(n)}(t)\|_\mathbb{H}^{2q}+[\int_0^t\mathcal{N}_1(x^{(n)}(s))\d s]^q
\le  C\|\psi(0)\|_\mathbb{H}^{2q}+C\int_0^t(1+\|x^{(n)}_s\|_\mathfrak{C}^{2q})
\d s+C\mathcal{N}(\psi)^q+C|M_n^{(1)}(t,x^{(n)})|^q,\end{equation}
where $C$ depends on $q$ and $t$ and independent of $n$.
For every given $n\in\mathbb{N}$ we define the stopping time
$$  \tau_R^{(n)}=\inf\{t\in[0,T]: \|x^{(n)}(t)\|_\mathbb{H}>R   \}\wedge T , \ R>0. $$
Here  $\inf\emptyset:=\infty$. It is obvious that
$$ \lim_{R\rightarrow\infty} \tau_R^{(n)}=T,\  n\in\mathbb{N}.  $$

By the B-D-G inequality we have
\begin{equation}\begin{split}\label{estimate 3}
  & \E^{P_n}\sup_{r\in[0,t\wedge\tau_R^{(n)}]}\left|M_n^{(1)}(r,x^{(n)})  \right|^q \\
   \le & C(q)\E^{P_n}\left(\int_0^{t\wedge\tau_R^{(n)}}\|x^{(n)}(s)\|_\mathbb{H}^2 \|B(s,x_s^{(n)})\|_{L_2(\mathbb{U},\mathbb{H})}^2 \d s  \right)^{q/2}\\
   \le & C(q)\E^{P_n}\left( \int_0^{t\wedge\tau_R^{(n)}} \|x^{(n)}(s)\|_\mathbb{H}^2\left(\|x_s^{(n)}\|_\mathfrak{C}^2+1\right) \d s  \right)^{q/2}\\
   \le & C(q) \E^{P_n} \left[ \varepsilon \sup_{s\in[0,t\wedge\tau_R^{(n)}]}\|x^{(n)}(s)\|_\mathbb{H}^{2q}+ C_\varepsilon \left( \int_0^{t\wedge\tau_R^{(n)}} (\|x_s^{(n)}\|_\mathfrak{C}^2+1) \d s  \right)^{q}    \right] \\
\le &  \varepsilon \E^{P_n}\sup_{s\in[0,t\wedge\tau_R^{(n)}]}\|x^{(n)}(s)\|_\mathbb{H}^{2q}+
t^{q -1} C_\varepsilon \E^{P_n}\int_0^{t\wedge\tau_R^{(n)}} \left(\sup_{r\in[0,s]}\|x^{(n)}(r)\|_\mathbb{H}^{2q}+
1 \right)\d s+C_T\|\psi\|_\mathfrak{C}^{2q},\ t\in[0, T],
\end{split}
\end{equation}
where $\varepsilon>0$ is a small constant and $C_\varepsilon$ comes from  Young's
inequality and may change from line to line.

 Then by   $(\ref{Ito})$,  $(\ref{estimate 3})$  and Gronwall's
lemma   we have
$$\E^{P_n}\sup_{t\in[0, \tau_R^{(n)}]}\|x^{(n)}(t)\|_\mathbb{H}^{2q} +\E^{P_n}[\int_0^{\tau_R^{(n)}}  \mathcal{N}_1(x^{(n)}(s))\d s]^q
\le C\left(\|\psi\|_\mathfrak{C}^{2q}+ 1+\mathcal{N}(\psi)^q\right), \ n\ge 1,$$
where $C$ is a constant independent of $n$.

For $R\rightarrow\infty$, (\ref{l2}) now follows from the monotone
convergence theorem.
\end{proof}

\textbf{Proof of Theorem \ref{T1}.}
Since $\Omega^{(n)}=C([-h,\infty);\mathbb{H}_n)$ is a closed subset of $\Omega$. We extend $P_n$ to a probability measure
$\hat{P}_n$ on $(\Omega,\mathcal{F})$ by setting
$$\hat{P}_n(A):=P_n(A\cap \Omega_n), A\in\mathcal{F}.$$
We now show that $(\hat{P}_n)_{n\in\mathbb{N}}$ is tight on
$$\mathbb{S}:=\mathbb{S}_0\cap L^{p,w}_{\rm{loc}}([-h,\infty);V),$$  where $\mathbb{S}_0:=C([-h,\infty),\mathbb{X}^*)\cap L^p_{\rm{loc}}([-h,\infty);\mathbb{Y})$ and $L^{p,w}_{\rm{loc}}([-h,\infty);V)$ denotes the space $L^{p}_{\rm{loc}}([-h,\infty);V)$ equipped with the weak topology.

 By [GRZ09, Lemma 4.3], the reflexivity of  space $L^p(-h,T;V)$ for every $T$, the fact that $\mathcal{N}(\psi)<\infty$, and by Lemma 3.1 and (\ref{bd8}) we only need to prove that for some $\beta>0$ and every $T>0$
$$\sup_{n\in\mathbb{N}}\E^{\hat{P}_n}\left(\sup_{s,t\in[0,T],s\neq t}\frac{\|x(t)-x(s)\|_{\mathbb{X}^*}}{|t-s|^\beta}\right)=
\sup_{n\in\mathbb{N}}\E^{P_n}\left(\sup_{s,t\in[0,T],s\neq t}\frac{\|x^{(n)}(t)-x^{(n)}(s)\|_{\mathbb{X}^*}}{|t-s|^\beta}\right)<\infty.$$
By (H3) and Lemma 3.1 we have
\begin{equation}\label{bd2}\aligned&\E^{P_n}\left[\sup_{s,t\in[0,T],s\neq t}\left(\|\int_s^t\Pi_nA_1(r,x^{(n)}(r))+\Pi_nA_2(r,x^{(n)}_r)\d r\|_{\mathbb{X}^*}^\gamma/|t-s|^{\gamma-1}\right)\right]
\\\leq&
C_\gamma\E^{P_n}\left[\int_0^T\|\Pi_nA_1(r,x^{(n)}(r))\|_{\mathbb{X}^*}^\gamma+\|\Pi_nA_2(r,x^{(n)}_r)\|_{\mathbb{X}^*}^\gamma \d r\right]\\\leq& C_\gamma
\E^{P_n}\left[\int_0^T\|A_1(r,x^{(n)}(r))\|_{\mathbb{X}^*}^\gamma+\|A_2(r,x^{(n)}_r)\|_{\mathbb{X}^*}^\gamma \d r\right]\\\leq& C_\gamma \E^{P_n}[\int_{0}^T\mathcal{N}_1(x^{(n)}(r))\d r]^{\gamma'}+C_\gamma\E^{P_n}(\sup_{r\in[-h,T]}\|x^{(n)}(r)\|_\mathbb{H}^{\gamma'}+1)+C_\gamma\mathcal{N}(\psi)^{\gamma'}\leq C_{T,\gamma'},\endaligned\end{equation}
and for every $0\leq s<t\leq T$ and $q\in\mathbb{N}$
$$\aligned\E^{P_n}\|M_n(t,x^{(n)})-M_n(s,x^{(n)})\|_{\mathbb{H}}^{2q}\leq & C_q\E^{P_n}\left(\int_s^t\|B(r,x_r)\|_{L_2(U;\mathbb{H})}^2\d r\right)^q\\\leq&C_q|t-s|^{q-1}
\E^{P_n}\left(\int_s^t\|B(r,x_r)\|_{L_2(U;\mathbb{H})}^{2q}\d r\right)\\\leq&C_q|t-s|^{q}
\E^{P_n}(\sup_{s\in[-h,T]}\|x(s)\|_\mathbb{H}^{2q}+1)\\\leq&C_q|t-s|^{q}
(\|\psi\|_\mathfrak{C}^{2q}+\mathcal{N}(\psi)^q+1).\endaligned$$
By Kolmogorov's criterion, for every $\alpha\in (0,\frac{q-1}{2q})$ we get
$$\E^{P_n}\left(\sup_{s,t\in[0,T],s\neq t}\frac{\|M_n(t,x^{(n)})-M_n(s,x^{(n)})\|_{\mathbb{H}}^{2q}}{|t-s|^{q\alpha}}\right)\leq C_q
(\|\psi\|_\mathfrak{C}^{2q}+\mathcal{N}(\psi)^q+1).$$
Now we obtain that for some $\beta>0$
$$
\sup_{n\in\mathbb{N}}\E^{P_n}\left(\sup_{s,t\in[0,T],s\neq t}\frac{\|x^{(n)}(t)-x^{(n)}(s)\|_{\mathbb{X}^*}}{|t-s|^\beta}\right)<\infty,$$
which as we mentioned above implies that $(\hat{P}_n)_{n\in\mathbb{N}}$ is tight on $\mathbb{S}$.

By Skorohod's representation theorem in [J98], there exists a probability space $(\tilde{\Omega},\tilde{\mathcal{F}},\tilde{P})$ and $\mathbb{S}$-valued random variables $\tilde{x}^{(n)}$ and $\tilde{x}$ such that

(i) $\tilde{x}^{(n)}$ has the law $\hat{P}_n$ for each $n\in\mathbb{N}$;

(ii) $\tilde{x}^{(n)}\rightarrow\tilde{x}$ in $\mathbb{S}$, $\tilde{P}$-a.e., and $\tilde{x}$ has the law $P$.
(Here selecting a subsequence if necessary.)

Then we have
$$\aligned &P(x(t)=\psi(t),-h\leq t\leq0)=\tilde{P}(\tilde{x}(t)=\psi(t),-h\leq t\leq0)\\=&\lim_{n\rightarrow\infty}\tilde{P}(\tilde{x}^{(n)}(t)=\Pi_n\psi(t),-h\leq t\leq0)=\lim_{n\rightarrow\infty}\hat{P}_n(x(t)=\Pi_n\psi(t),-h\leq t\leq0)=1.\endaligned$$

For every $q\in\mathbb{N}$  set
$$\xi_q(t,x):=\sup_{r\in [-h,t]}\|x(r)\|_{\mathbb{H}}^{2q}+\left(\int_{0}^t\mathcal{N}_1(x(r))\d r\right)^q.$$
By Fatou's lemma, Lemma 3.1 and the lower semi-continuity of $\xi_q(t,x)$ with respect to $x$ on $\mathbb{S}$, we have
$$\aligned\E^P(\xi_q(t,x))=&\E^{\tilde{P}}[\xi_q(t,\tilde{x})]\leq \liminf_{n\rightarrow\infty}\E^{\tilde{P}}[\xi_q(t,\tilde{x}^{(n)})]
\\=&\liminf_{n\rightarrow\infty}\E^{\hat{P}_n}[\xi_q(t,x)]\\\leq&C
[\|\psi\|_\mathfrak{C}^{2q}+1+\mathcal{N}(\psi)^q].\endaligned$$

Now we verify (M2) for $P$.

Fixing $l\in\mathcal{E}$, we want to show $M_l(t,x),t\geq0,$ in (M2) is a continuous $(\mathcal{F}_t)$-martingale with respect to $P$, whose square variation process is given by
$$\langle M_l\rangle(t,x)=\int_0^t\|B^*(s,x_s)(l)\|^2_U\d s,t\geq0.$$
By Lemma 3.1 we have
\begin{equation}\label{bd1}\lim_{n\rightarrow\infty}\E^{\tilde{P}}|{ }_{\mathbb{X}^*}\langle \tilde{x}^{(n)}(t)-\tilde{x}(t),l\rangle_{\mathbb{X}}|=0.\end{equation}
Now for the nonlinear part,
set for $R>0$
$$G_R^1(t,x):=\int_0^t{ }_{\mathbb{X}^*}\langle A_1(s,x(s)),l\rangle_{\mathbb{X}}\cdot \chi_R(\|x(s)\|_\mathbb{H})\d s,$$
where $\chi_R\in C_0^\infty(\mathbb{R})$ is a cutoff function with $\chi_R(r)=1$ if $|r|\leq R$ and $\chi_R(r)=0$ if $|r|>2R$.
By (H1) we obtain $\tilde{P}$-a.s.
\begin{equation}\label{co2}\int_0^t{ }_{\mathbb{X}^*}\langle A_2(s,\tilde{x}_s^{(n)}),l\rangle_{\mathbb{X}}\d s\rightarrow\int_0^t{ }_{\mathbb{X}^*}\langle A_2(s,\tilde{x}_s),l\rangle_{\mathbb{X}}\d s, \quad n\rightarrow\infty, \end{equation}
and
$$G_R^1(t,\tilde{x}^{(n)})\rightarrow G_R^1(t,\tilde{x}),\quad n\rightarrow\infty.$$
Moreover, by (H3) and Lemma 3.1 we have
$$\aligned&\sup_n\E^{\tilde{P}}\left(\int_0^t\| A_1(s,\tilde{x}^{(n)}(s))\|_{\mathbb{X}^*}^\gamma \d s\right)\\=&\sup_n\E^{P_n}\left(\int_0^t\| A_1(s,x^{(n)}(s))\|_{\mathbb{X}^*}^\gamma \d s\right)<\infty,\endaligned$$
and $$\aligned&\sup_n\E^{\tilde{P}}\left(\int_0^t\| A_2(s,\tilde{x}_s^{(n)})\|_{\mathbb{X}^*}^\gamma \d s\right)\\=&\sup_n\E^{P_n}\left(\int_0^t\| A_2(s,x_s^{(n)})\|_{\mathbb{X}^*}^\gamma \d s\right)<\infty,\endaligned$$
which  implies that
$$\lim_{n\rightarrow\infty}\E^{\tilde{P}}|G_R^1(t,\tilde{x}^{(n)})-G_R^1(t,\tilde{x})|=0,$$
and
\begin{equation}\label{bd3}\lim_{n\rightarrow\infty}\E^{\tilde{P}}|\int_0^t[{ }_{\mathbb{X}^*}\langle A_2(s,\tilde{x}^{(n)}_s),l\rangle_{\mathbb{X}}-{ }_{\mathbb{X}^*}\langle A_2(s,\tilde{x}_s),l\rangle_{\mathbb{X}}]\d s|=0.\end{equation}
Set
$$G(t,x):=\int_0^t{ }_{\mathbb{X}^*}\langle A_1(s,x),l\rangle_{\mathbb{X}}\d s,$$
by (H3) and Lemma 3.1 we have
$$\aligned&\lim_{R\rightarrow\infty}\sup_n\E^{\tilde{P}}|G_R^1(t,\tilde{x}^{(n)})-G(t,\tilde{x}^{(n)})|
\\\leq&\|l\|_{\mathbb{X}}\lim_{R\rightarrow\infty}\sup_n[(\E^{\tilde{P}}\int_0^t
\|A_1(s,\tilde{x}^{(n)}(s))\|^\gamma_{\mathbb{X}^*}\d s)^{1/\gamma} (\int_0^t\tilde{P}(\|\tilde{x}^{(n)}(s)\|_\mathbb{H}\geq R)\d s)^{(\gamma-1)/\gamma}]\\\leq&\|l\|_{\mathbb{X}}\lim_{R\rightarrow\infty}\sup_n(
\E^{\tilde{P}}\int_0^t\|A_1(s,\tilde{x}^{(n)}(s))\|^\gamma_{\mathbb{X}^*}\d s)^{1/\gamma}(\E^{\tilde{P}}\int_0^t\|\tilde{x}^{(n)}(s)\|_\mathbb{H}^\gamma\d s)^{(\gamma-1)/\gamma}/R^{\gamma-1}\\=&\|l\|_{\mathbb{X}}\lim_{R\rightarrow\infty}\sup_n(\E^{P_n}\int_0^t
\|A_1(s,x^{(n)}(s))\|^\gamma_{\mathbb{X}^*}\d s)^{1/\gamma}(\E^{P_n}\int_0^t\|x^{(n)}(s)\|_\mathbb{H}^\gamma\d s)^{(\gamma-1)/\gamma}/R^{\gamma-1}=0,\endaligned$$
and similarly
$$\aligned&\lim_{R\rightarrow\infty}\E^{\tilde{P}}|G_R^1(t,\tilde{x})-G(t,\tilde{x})|
=0,\endaligned$$
which  imply that
\begin{equation}\label{co4}\lim_{n\rightarrow\infty}\E^{\tilde{P}}|G(t,\tilde{x}^{(n)})-G(t,\tilde{x})|=0.\end{equation}

On the other hand, by Lemma 3.1 and (H3) we have
\begin{equation}\label{bd4}\aligned&\lim_{n\rightarrow\infty}\E^{\tilde{P}}|\int_0^t{ }_{\mathbb{X}^*}\langle \Pi_nA_1(s,\tilde{x}^{(n)}(s))+\Pi_nA_2(s,\tilde{x}^{(n)}_s),l\rangle_{\mathbb{X}}-{ }_{\mathbb{X}^*}\langle A_1(s,\tilde{x}^{(n)}(s))+A_2(s,\tilde{x}^{(n)}_s),l\rangle_{\mathbb{X}}\d s|\\\leq&\lim_{n\rightarrow\infty}\E^{\tilde{P}}|\int_0^t{ }_{\mathbb{X}^*}\langle A_1(s,\tilde{x}^{(n)}(s))+A_2(s,\tilde{x}^{(n)}_s),(\Pi_n-I)l\rangle_{\mathbb{X}}\d s|=0.\endaligned\end{equation}
Combining (\ref{bd1}), (\ref{bd3}), (\ref{co4}) and (\ref{bd4}) we obtain that for $t>0$
\begin{equation}\label{bd5}\lim_{n\rightarrow\infty}\E^{\tilde{P}}|\langle M_n(t,\tilde{x}^{(n)}),l\rangle-M_l(t,\tilde{x})|=0.\end{equation}
Let $t>s\geq0$ and $g$ be every bounded and real valued $\mathcal{F}_s$-measurable continuous function on $\mathbb{S}$. Using (\ref{bd5}) we have
$$\aligned\E^P((M_l(t,x)-M_l(s,x))g(x))
=&\E^{\tilde{P}}((M_l(t,\tilde{x})-M_l(s,\tilde{x}))g(\tilde{x}))\\
=&\lim_{n\rightarrow\infty}
\E^{\tilde{P}}((\langle M_n(t,\tilde{x}^{(n)}),l\rangle-\langle M_n(s,\tilde{x}^{(n)}),l\rangle)g(\tilde{x}^{(n)}))\\=&\lim_{n\rightarrow\infty}
\E^{\hat{P}_n}((\langle M_n(t,x),l\rangle-\langle M_n(s,x),l\rangle)g(x))\\=&0.\endaligned$$

On the other hand by the B-D-G inequality  (H3) and Lemma 3.1 we have
$$\aligned\sup_n\E^{\tilde{P}}|\langle M_n(t,\tilde{x}^{(n)}),l\rangle|^{2q}\leq& C\sup_n\E^{\tilde{P}}\left(\int_0^t \|B^*(s,\tilde{x}^{(n)}_s)(l)\|_U^2\d s\right)^{q}\\\leq& C\sup_n\E^{\tilde{P}}\left(\int_0^t \|B^*(s,\tilde{x}^{(n)}_s)(l)\|_U^{2q}\d s\right)<+\infty.\endaligned$$
Then by (\ref{bd5}) we obtain that for $t\geq0$
$$\lim_{n\rightarrow\infty}\E^{\tilde{P}}|\langle M_n(t,\tilde{x}^{(n)}),l\rangle-M_l(t,\tilde{x})|^2=0,$$
and by (H1)
$$\lim_{n\rightarrow\infty}\E^{\tilde{P}}\left(\int_0^t \|(\Pi_nB\tilde{\Pi}_n)^*(s,\tilde{x}^{(n)}_s)(l)-B^*(s,\tilde{x}_s)(l)\|_U^2\d s\right)=0,$$
which imply that
$$\E^P\left(M_l^2(t,x)-\int_0^t\|B^*(r,x_r)(l)\|^2_{U}\d r|\mathcal{F}_s\right)=M_l^2(s,x)-\int_0^s\|B^*(r,x_r)(l)\|^2_{U}\d r.$$
Now the result follows.
$\hfill\Box$

\section{Uniqueness of the solutions}
Now we consider the pathwise uniqueness of the solutions. We introduce another Hilbert space $H_0$, which is the space where we obtain uniqueness of the solutions.

 Let
$$V\subset H_0\cong H_0^*\subset V^*$$
 be a Gelfand triple, $i.e.$  $(H_0, \<\cdot,\cdot\>_{H_0})$ is a separable
Hilbert space, identified with its dual space by the Riesz
isomorphism, $V$ is as in Section 2 and it is
 continuously and densely embedded into $H_0$. If  ${ }_{V^*}\<\cdot,\cdot\>_V$ denotes  the dualization
between  $V$ and its dual space $V^*$,  it follows that
$$ { }_{V^*}\<u, v\>_V=\<u, v\>_{H_0}, \  u\in H_0 ,v\in V.$$
Since $\mathbb{X}\subset V$ continuously  we have that
 $$V^*\subset \mathbb{X}^*,$$ and
 there exists a bounded linear operator $\Lambda:\mathbb{X}\rightarrow V$  such that
\begin{equation}\label{eq5}{ }_{V^*}\<u, \Lambda v\>_V={ }_{\mathbb{X}^*}\<u,  v\>_{\mathbb{X}}, u\in V^*,v\in \mathbb{X}.\end{equation}
In fact, since $\mathbb{X}\subset V$ continuously, we have that for $u\in H_0\subset V^*\subset \mathbb{X}^*,v\in \mathbb{X}$
$$|{ }_{\mathbb{X}^*}\<u,  v\>_{\mathbb{X}}|\leq \|u\|_{\mathbb{X}^*}\|v\|_{\mathbb{X}}\leq C\|u\|_{V^*}\|v\|_{\mathbb{X}}\leq C\|u\|_{H_0}\|v\|_{\mathbb{X}}, $$
which implies that there exists some $\Lambda_0v\in H_0$ such that for $u\in H_0$
$$\<u, \Lambda_0 v\>_{H_0}={ }_{\mathbb{X}^*}\<u,  v\>_{\mathbb{X}},$$
hence $v\mapsto\Lambda_0 v$ is linear.
Moreover since $$|\<u, \Lambda_0 v\>_{H_0}|=|{ }_{\mathbb{X}^*}\<u,  v\>_{\mathbb{X}}|\leq C\|u\|_{V^*},$$
by extension (\ref{eq5}) follows.

Suppose that for $p\geq2$ as in Section 2, $A_1:\mathbb{R}^+\times V\to V^*$ is $\mathcal{B}(\mathbb{R}^+)\otimes\mathcal{B}(V)$-measurable;
$A_2:\mathbb{R}^+\times (L_V^p\cap \mathfrak{C})\to V^*$ is $\mathcal{B}(\mathbb{R}^+)\otimes\mathcal{B}(L_V^p\cap \mathfrak{C})$-measurable; $B:
\mathbb{R}^+\times (L_V^p\cap \mathfrak{C})\to L_{2}(U;H_0)$ is $\mathcal{B}(\mathbb{R}^+)\otimes\mathcal{B}(L_V^p\cap \mathfrak{C})$-measurable.

 There exist locally bounded measurable functions $c_2, c_3:\mathbb{R}\rightarrow\mathbb{R}^+$ and $\rho_1: \mathbb{R}^+\times (\mathfrak{C}\cap L_V^p)\rightarrow[0,+\infty)$ such that for $ u ,\eta\in \mathbb{S}_1$,
$A_1(\cdot, u (\cdot)), A_2(\cdot, u _\cdot)\in L^1_{\rm{loc}}([0,\infty);V^*), B(\cdot, u _\cdot)\in L^2_{\rm{loc}}([0,\infty);L_2(U;H_0))$ and every $t\geq0$:

 (H4) (Local monotonicity)
$$\aligned&2\int_0^t{ }_{V^*}\!\langle A_1(s, u (s))+A_2(s, u _s)-A_1(s,\eta(s))-A_2(s,\eta_s), u (s)-\eta(s)\rangle_V\d s
\\\leq& \int_{0}^t[\rho_1(s,\eta_s)+\rho_1(s, u _s)]\| u _s-\eta_s\|_{C([-h,0];H_0)}^2\d s+c_2(t)\int_{-h}^0\mathcal{N}_1( u (s)-\eta(s))\d s, \endaligned$$
$$\aligned&\int_0^t\|B(s, u _s)-B(s,\eta_s)\|_{L_2(U;H_0)}^2\d s\\\leq& \int_{0}^t[\rho_1(s,\eta_s)+\rho_1(s, u _s)]\| u _s-\eta_s\|_{C([-h,0];H_0)}^2\d s+c_2(t)\int_{-h}^0\mathcal{N}_1( u (s)-\eta(s))\d s, \endaligned$$
and
$$\int_{0}^t\rho_1(s,\eta_s)\d s\leq c_3(t)[\int_{-h}^t\mathcal{N}_1(\eta(s))\d s]^q+c_3(t)(1+\sup_{s\in[-h,t]}\|\eta(s)\|_\mathbb{H}^{2q}),$$
for some $q\in\mathbb{N}$.

(H5) (Growth condition) There exist locally bounded measurable functions $\lambda_5, \lambda_6:\mathbb{R}\mapsto\mathbb{R}^+$ such that for  $ u \in  \mathbb{S}_1$, we have for  every $t\geq0$,
$$\int_0^t[\|A_1(s, u (s))\|_{V^*}^{\frac{p}{p-1}}+\|A_2(s, u _s)\|_{V^*}^{\frac{p}{p-1}}]\d s\leq  \lambda_5(t)[\int_{-h}^t\mathcal{N}_1( u (s))\d s]^q+\lambda_5(t)(\sup_{s\in[-h,t]}\| u (s)\|_\mathbb{H}^{2q}+1),$$
$$\|B(t, u _t)\|_{L_2(U;H_0)}^{2}\leq \lambda_6(t)(1+\| u _t\|^{2}_\mathfrak{C}),$$
for some $q\in\mathbb{N}$.

By (H5) and Theorem \ref{T2} we obtain that
\beg{thm}\label{T3}
 Suppose $(H1)-(H3), (H5)$ hold for some $p\geq2$.
    Then for every $\psi\in  C([-h,0];  H_0)\cap \mathfrak{C}$ with $\mathcal{N}(\psi)<\infty$,
    $(\ref{SEE})$
    has a weak solution $X\in C([-h,\infty);H_0)$ such that for every $q\in\mathbb{N}$, $T>0$
$$\E^P\left(\sup_{t\in[-h,T]}\|X(t)\|_\mathbb{H}^{2q}+[\int_{-h}^T\mathcal{N}_1(X(t))\d t]^q \right)   < \infty,$$
and
$$\E^P\left(\sup_{t\in[-h,T]}\|X(t)\|_{H_0}^{2}\right)\leq C.$$

\end{thm}

\proof   By Theorem 2.2 we have that for every $l\in \mathbb{X}$, $t\in[0,T]$
$${ }_{\mathbb{X}^*}\langle X(t),l\rangle_{\mathbb{X}}-\int_0^t{ }_{\mathbb{X}^*}\langle A_1(s,X(s))+A_2(s,X_s),l\rangle_{\mathbb{X}} \d s
=\langle X(0),l\rangle_\mathbb{H}+\langle \int_0^tB(s,X_s)\d W(s),l\rangle_\mathbb{H}.$$
By (H5) and Theorem 2.2 we know that for every $T>0$
\begin{equation}\label{bd6}\E^P\int_0^T[\|A_1(s,X(s))\|_{V^*}^{\frac{p}{p-1}}+\|A_2(s,X_s)\|_{V^*}^{\frac{p}{p-1}}]\d s\leq C, \end{equation}
\begin{equation}\label{bd7}\E^P\int_0^T \|B(s,X_s)\|_{L_2(U;H_0)}^2\d s\leq C,\end{equation}
for some constant $C$. We define
$$\bar{X}(t):=X(0)+\int_0^t[A_1(s,X(s))+A_2(s,X_s)]\d s+\int_0^tB(s,X_s)\d W(s),$$
in $V^*$.
By (\ref{eq5}) we obtain that for every $l\in\mathbb{X}$, $t\in[0,T]$
$${ }_{\mathbb{X}^*}\langle \bar{X}(t),l\rangle_{\mathbb{X}}-\int_0^t{ }_{\mathbb{X}^*}\langle A_1(s,X(s))+A_2(s,X_s),l\rangle_{\mathbb{X}} \d s
=\langle X(0),l\rangle_\mathbb{H}+\langle \int_0^tB(s,X_s)\d W(s),l\rangle_\mathbb{H},$$
which implies that
$$X(t)=\bar{X}(t) \quad \forall t\in [0,\infty), P-a.s..$$
Moreover, by (\ref{bd8}) and Theorem 2.2 we obtain
$$\E^P\int_0^T\|X(s)\|_V^p\d s\leq C,$$
which combined with (4.2), (\ref{bd7}) and [PR07, Theorem 4.2.5] implies that
$X\in C([-h,\infty),H_0)$ and
$$\E^P[\sup_{t\in[0,T]}\|X(t)\|_{H_0}^2]\leq C.$$$\hfill\Box$

\begin{defn}  We say that there exists a (probabilistically) strong solution to (2.1)  if for every probability space $(\Omega,\mathcal{F},(\mathcal{F}_t)_{t\in [0,T]},P)$ with an $(\mathcal{F}_t)$-Wiener process $W$ and given an initial value $\psi\in L^p(-h,0;V)\cap C([-h,0];H_0)\cap \mathfrak{C} $ there exists  an $(\mathcal{F}_t)$-adapted process $X:[0,T]\times \Omega\rightarrow H_0$ such that
 for  $P$-a.e. $\omega\in \varOmega$,
$$X(\cdot,\omega)\in \mathbb{S}_1\cap C([-h,\infty);H_0),$$  and such that $P$-a.s.
$$X(t)=\psi(t), -h\leq t\leq 0,$$
$$X(t)
= X(0)+\int_0^t[ A_1(s,X(s))+A_2(s,X_s)] \d s+\int_0^tB(s,X_s)\d W(s),$$
holds in $V^*$ for all $t\in [0,\infty)$.
\end{defn}

\beg{thm}\label{T4}
 Suppose that $(H1)-(H5)$ hold.
    Then for every $\psi\in   C([-h,0];  H_0)\cap \mathfrak{C}$ with $\mathcal{N}(\psi)<\infty$,
    $(\ref{SEE})$
    has a unique (probabilistically) strong solution $X\in C([-h,\infty);H_0)$ satisfying for every  $T>0$ and $q\in\mathbb{N}$
\begin{equation}\label{bd6}\E^P\left(\sup_{t\in[-h,T]}\|X(t)\|_\mathbb{H}^{2q}+[\int_0^T\mathcal{N}_1(X(t))\d t]^q \right)   < \infty.\end{equation} Moreover,
$$\E^P\left(\sup_{t\in[-h,T]}\|X(t)\|_{H_0}^{2}\right)<\infty.$$

\end{thm}

\proof Suppose $X,Y$ are the solutions of (\ref{SEE}), both with initial conditions
$\psi$ respectively, i.e.
 \begin{equation}\begin{split}
 X(t)&=\psi(0)+\int_0^t [A_1(s,X(s))+A_2(s,X_s)] \d s+\int_0^t B(s,X_s) \d W(s), \ t\in[0,T];\\
  Y(t)&=\psi(0)+\int_0^t [A_1(s,Y(s))+A_2(s,Y_s)] \d s+\int_0^t B(s,Y_s) \d W(s), \ t\in[0,T].
\end{split}
\end{equation}
Then by  It\^{o}'s formula  we have
 \begin{equation*}\begin{split}
 & \|X(t)-Y(t)\|_{H_0}^2\\
\le & 2\int_0^t   { }_{V^*}\< A_1(s,X(s))+A_2(s,X_s)-A_1(s,Y(s))-A_2(s,Y_s),  X(s)-Y(s)\>_V\d s
\\&+2\int_0^t   \<X(s)-Y(s), (B(s,X_s)-B(s,Y_s))\d W(s)  \>_{H_0}\\&+\int_0^t\|B(s,X_s)-B(s,Y_s)\|_{L_2(U;{H_0})}^2\d s,\ t\in[0,T].
\end{split}
\end{equation*}
Define $$\tau_n:=\inf\{t\geq0| \int_{0}^t[\rho_1(s,Y_s)+\rho_1(s,X_s)]\d s>n\}\wedge T.$$
By (\ref{bd6}) and (H4) we have
$$\lim_{n\rightarrow\infty}\tau_n=T, \quad P-a.s..$$

By (H4) and the B-D-G inequality we have for every stopping time $\tau$
 \begin{equation*}\begin{split}
 & \E\sup_{t\in[0,\tau\wedge\tau_n]}\|X(t)-Y(t)\|_{H_0}^2\\
\le & C\E\int_{0}^{\tau\wedge\tau_n}   (\rho_1(s,X_s)+\rho_1(s,Y_s))\|X_s-Y_s\|_{C([-h,0];{H_0})}^2\d s
\\&+2\E\sup_{0\leq t\leq \tau\wedge\tau_n}\int_0^t   \<X(s)-Y(s), (B(s,X_s)-B(s,Y_s))\d W(s)  \>_{H_0}\\
\le & C\E\int_{0}^{\tau\wedge\tau_n}   (\rho_1(s,X_s)+\rho_1(s,Y_s))\|X_s-Y_s\|_{C([-h,0];{H_0})}^2\d s+\frac{1}{2}\E\sup_{t\in[0,\tau\wedge\tau_n]}\|X(t)-Y(t)\|_{H_0}^2.
\end{split}
\end{equation*}

By [GM83 Lemma 2]
we have for every $n\in\mathbb{N}$
$$  X(t)=Y(t),  \ t\in[0,\tau_n],  \ \P-a.s.. $$
Letting $n\rightarrow\infty$ we have
$$  X(t)=Y(t),  \ t\in[0,\infty),  \ \P-a.s.. $$
Then by using the Yamada-Watanabe Theorem (see e.g. [Ku07], [RSZ08]) the results follow.

\qed

 \section{Application to examples}
In this section we describe some examples for (\ref{SEE}) satisfying conditions (H1)-(H5) imposed above. First, we recall some useful estimates which will be used later.

Let $O$ be a bounded open domain in $\mathbb{R}^d$ with smooth boundary and let $C^\infty_0(O)$  denote the set of all smooth functions from
$O$ to $\mathbb{R}$ with compact support. For $p > 1$, let $L
^p(O)$ be the  $L^p$
-space with  norm  denoted by $\|\cdot\|_{L^p}$. If $A$ is $-\Delta$ on the domain $O$ with Dirichlet boundary condition, then we have the following estimates which will be used later. For $s \geq0, p \in[1,\infty]$ we use $W_0^{s,p}(O)$  to denote the Sobolev space of all $f\in L^2$ for which $\|A^{s/2}f\|_{L^p}$
is finite.
\vskip.10in
\begin{rem}Here we can also consider the equation on the torus $\mathbb{T}^d$ and the results still can be applied in this case. \end{rem}

We shall also use the following standard Sobolev inequality (cf. [St70, Chapter V]):
\vskip.10in
\begin{lem} Suppose that $q>1, p\in [q,\infty)$ and
$$\frac{1}{p}+\frac{\sigma}{d}=\frac{1}{q}.$$
Suppose that $f\in W^{\sigma,q}_0$, then $f\in L^p$ and there is a constant $C\geq 0$ such that
$$\|f\|_{L^p}\leq C\| f\|_{W^{\sigma,q}_0}.$$
\end{lem}
\vskip.10in

In the following, let $\{W^k(t);t\geq0, k\in\mathbb{N}\}$ be a sequence of independent standard Brownian motions and $l^2$ be the Hilbert space of all square summable sequences of real number sequences. We will use $\partial_i:=\partial_{\xi_i}, \partial_{ij}^2:=\partial^2_{\xi_i\xi_j}$ and the usual Einstein summation convention.

\subsection{Fractional Navier-Stokes equation}

Now we apply Theorem \ref{T1} to the stochastic d-dimensional fractional Navier-Stokes equation in  a bounded domain $O$ in $\mathbb{R}^d$ with smooth boundary:
\begin{equation}\label{NS}\aligned du(t)=&[-(-\Delta)^\alpha u(t)-u(t)\cdot\nabla u(t)+\nabla p(t)\\&+c_1u(t-r_1(t))\cdot \nabla u(t)+c_2\int_{-h}^0u(t+r)\cdot \nabla u(t)dr\\&+\int_{-h}^0f_1(r,\xi,u(t+r))dr+f(\xi,u(t-r_2(t)))]dt\\&+[\nabla \tilde{p}_i(t)+g_i(\xi,u(t-r_3(t)))]dW^i(t),\endaligned\end{equation}
$$\rm{div}u(t)=0,$$
$$u(t,\xi)=0, (t,\xi)\in[-h,\infty)\times \partial O,$$
$$u(t)=\psi(t), -h\leq t\leq 0,$$
where $p(t,\xi)$ and $\tilde{p}_i(t,\xi), i\in\mathbb{N},$ are unknown scalar functions, $u$ is the velocity vector, $c_i\in\mathbb{R}, i=1,2$,  $r_i, i=1,2,3,$  are continuous differentiable  functions on $\mathbb{R}$ with $0\leq r_i(t)<h, i=1,2,3, \sup r_i'(t)<1, i=1,2$,  $f$ and $g$ are functions from $O\times \mathbb{R}^d$ to $\mathbb{R}^d$ and $\mathbb{R}^d\times l^2$, respectively, continuous with respect to the second variable, $f_1$ is a function from $[-h,0]\times O\times \mathbb{R}^d$ to $\mathbb{R}^d$, continuous with respect to the third variable,  satisfying for some $\kappa_0>0$ and $q\in L^2(O)$
$$|f_1(r,\xi,u)|+|f(\xi,u)|+\|g(\xi,u)\|_{l^2}\leq \kappa_0 \cdot|u|+q(\xi), \quad \forall (r,\xi,u)\in [-h,0]\times O\times \mathbb{R}^d.$$

When $\alpha=1,g=c_1=c_2=f=f_1=0$, (5.1) reduces to the usual Navier-Stokes equation. In the  case without delay for general $\alpha$ this equation has been studied by many authors (see [W05], [Z12], [RZZ14b] and the references therein). In [W05], the author  obtains global existence and uniqueness of solutions for small initial values when $\alpha>\frac{1}{2}$. In [Z12], the author obtains  local existence and uniqueness of solutions by using a stochastic Lagrangian particle
trajectories approach for $\alpha=\frac{1}{2}$ and global existence and uniqueness when $d=2$. In [RZZ14b] we obtain  local existence and uniqueness of solutions of stochastic fractional Navier-Stokes equations for $\alpha\in(0,1]$ and every $d\in\mathbb{N}$.

Let $C_{0,\sigma}^\infty(O)^d$ be the space of all smooth $d$-dimensional divergence free vector fields on $O$ with compact supports in $O$. For $s\geq0, p>1$, the completion of $C_{0,\sigma}^\infty(O)^d$ in $W_0^{s,p}(O)^d$ is denoted by $W_{0,\sigma}^{s,p}(O)$.
Below we choose
$$\mathbb{Y}=\mathbb{H}=W^{0,2}_{0,\sigma}(O), V=W^{\alpha,2}_{0,\sigma}(O)$$
and
$$\mathbb{X}^*=(W^{2+d,2}_{0,\sigma}(O))^*, \mathbb{X}=W^{2+d,2}_{0,\sigma}(O). $$
Let  $\mathcal{P}$ be the orthogonal projection operator from $L^2(O)^d$ onto $\mathbb{H}$. We define the operators $A_1, A_2$ and $B$ as follows: for $x\in C_{0,\sigma}^\infty(O)$
$$A_1(x):=-\mathcal{P}[(-\Delta)^\alpha x]-\mathcal{P}[(x\cdot\nabla)x],$$
and for $u\in C([-h,0];C_{0,\sigma}^\infty(O))$
$$\aligned A_2(t,u):=&c_1\mathcal{P}[u(-r_1(t))\cdot \nabla u(0)]+c_2\mathcal{P}[\int_{-h}^0u(r)\cdot \nabla u(0)\d r]\\&+\mathcal{P}[\int_{-h}^0f_1(r,u(r))\d r]
+\mathcal{P}f(u(-r_2(t))).\endaligned$$
$$B(t,u):=\mathcal{P}[g(u(-r_3(t)))].$$

Then by similar arguments as in [GRZ09] we have that for $x_1,y_1,x,y\in C_{0,\sigma}^\infty(O)$,
\begin{equation}\|\mathcal{P}(-\Delta)^\alpha x-\mathcal{P}(-\Delta)^\alpha y\|_{\mathbb{X}^*}\leq C\|x-y\|_\mathbb{H},\end{equation}
\begin{equation}\|\mathcal{P}[x_1\cdot\nabla x]-\mathcal{P}[y_1\cdot\nabla y]\|_{\mathbb{X}^*}\leq C\|x_1\|_\mathbb{H}\|x-y\|_\mathbb{H}+C\|y\|_\mathbb{H}\|x_1-y_1\|_\mathbb{H}.\end{equation}
We can extend the operators $A_1$ and $A_2, B$ to $\mathbb{H}$ and $\mathfrak{C}$ such that
for $x\in \mathbb{H}, u\in \mathfrak{C}$, $A_1(x)\in \mathbb{X}^*$ and $A_2(t,u)\in \mathbb{X}^*$, $B(t,u)\in L_2(l_2,\mathbb{H})$.
Thus, we can write the system in the following abstract form
$$\d u(t)=[A_1(t,u(t))+A_2(t,u_t)]\d t+B(t,u_t)\d W(t), \quad u(t)=\psi(t), -h\leq t\leq0.$$
Moreover, by (5.2), (5.3) we have that for $u\in \mathfrak{C}_\infty$
$$A_1(u(\cdot))\in L^1_{\rm{loc}}([0,\infty);\mathbb{X}^*), A_2(\cdot,u_\cdot)\in L^1_{\rm{loc}}([0,\infty);\mathbb{X}^*).$$
Let $e_n,n\in\mathbb{N},$ $\Pi_n, \mathbb{H}_n$ be as in Section 2. We can easily deduce that
$\|\Pi_nA_1(\cdot)\|_{\mathbb{X}^*}$ is locally bounded and continuous on $\mathbb{H}_n$ and
$\|\Pi_nA_2(\cdot,\cdot)\|_{\mathbb{X}^*}, \|\Pi_nB(\cdot,\cdot)\|_{L_2(U;\mathbb{H})}$ are  bounded on balls in $\mathbb{R}^+\times C([-h,0];\mathbb{H}_n)$. Furthermore, for every $t\in[0,T]$, we have for $u,v\in C([-h,\infty);\mathbb{X})$
$$\aligned&\|\mathcal{P}[u(t-r_1(t))\cdot\nabla u(t)]-\mathcal{P}[v(t-r_1(t))\cdot\nabla v(t)]\|_{\mathbb{X}^*}
\\\leq& C[\sup_{s\in [-h,T]}\|u(s)\|_{\mathbb{H}}+\sup_{s\in [-h,T]}\|v(s)\|_{\mathbb{H}}]\sup_{s\in [-h,T]}\|u(s)-v(s)\|_{\mathbb{H}},\endaligned$$
and
$$\aligned&\|\int_{-h}^0\mathcal{P}[u(t+r)\cdot\nabla u(s)]-\mathcal{P}[v(t+r)\cdot\nabla v(t)]\d r\|_{\mathbb{X}^*}\\\leq& C[\sup_{s\in [-h,T]}\|u(s)\|_{\mathbb{H}}+\sup_{s\in [-h,T]}\|v(s)\|_{\mathbb{H}}]\sup_{s\in [-h,T]}\|u(s)-v(s)\|_{\mathbb{H}},\endaligned$$
which implies that for every $t>0$, $\Pi_n A_2(t,\cdot), \Pi_n B(t,\cdot)$ are continuous on $C([-h,0];\mathbb{H}_n)$.

We define the functional $\mathcal{N}_1$ on $\mathbb{H}$ as follows:
$$\mathcal{N}_1(u)=\left\{\begin{array}{ll}\|(-\Delta)^\frac{\alpha}{2} u\|_{L^2(O)}^2&\ \ \ \ \textrm{ if } u\in W^{\alpha,2}_{0,\sigma}(O),\\0&\ \ \ \ \textrm{ otherwise} .\end{array}\right.$$
We easily deduce that $\mathcal{N}_1\in\mathfrak{U}^2$. In this case $\mathcal{N}(\psi)<\infty$ is equivalent to $\psi\in L^2([-h,0];V)$.

\beg{thm}\label{TN}
 For every $\psi\in   \mathfrak{C}\cap L^2_V$
    $(\ref{NS})$
    has a weak solution $X:[-h,\infty)\times \Omega\rightarrow \mathbb{H}$ such that for every $q\in\mathbb{N}$, $T>0$
$$\E^P\left(\sup_{t\in[-h,T]}\|X(t)\|_\mathbb{H}^{2q}+[\int_{-h}^T\mathcal{N}_1(X(t))dt]^q \right)   < \infty.$$

\end{thm}

\proof First we check (H3): by (5.2), (5.3) and the linear growth of  $f, f_1,g$ we have for $u\in \mathfrak{C}_\infty$, $t\geq0$,
$$\|A_1(u(t))\|_{\mathbb{X}^*}\leq C(1+\|u_t\|_\mathfrak{C}^2),$$
$$\aligned\|A_2(t,u_t)\|_{\mathbb{X}^*}\leq& C[\|u(t-r_1(t))\|_\mathbb{H}\|u(t)\|_\mathbb{H}+\int_{-h}^0\|u(t+r)\|_\mathbb{H}\|u(t)\|_\mathbb{H}\d r]\\&+C[1+\|u(t-r_2(t))\|_\mathbb{H}+\int_{-h}^0\|u(t+r)\|_\mathbb{H}\d r]
\\\leq& C(1+\|u_t\|_\mathfrak{C}^2),\endaligned$$
$$\|B(t,u_t)\|^2_{L_2(l^2;\mathbb{H})}\leq  C(1+\|u_t\|_\mathfrak{C}^2),$$
where $C$ is a constant changing from line to line. Now (H3) follows immediately by integrating with respect to $t$.

Now we check (H1):  by  Remark 2.2, (H3), (5.2) and (5.3) it follows that (\ref{cd1}) holds in this case. Now by (5.3) we have for $u,v\in L^2_{\rm{loc}}([-h,\infty);\mathbb{H})$ and $t>0$
$$\aligned&\int_0^t\|\mathcal{P}[u(s-r_1(s))\cdot\nabla u(s)]-\mathcal{P}[v(s-r_1(s))\cdot\nabla v(s)]\|_{\mathbb{X}^*}\d s\\\leq&C(\int_0^t\|u(s)\|_\mathbb{H}^2\d s)^{\frac{1}{2}}(\int_0^t\|u(s-r_1(s))-v(s-r_1(s))\|_{\mathbb{H}}^2\d s)^{\frac{1}{2}}\\&+(\int_0^t\|v(s-r_1(s))\|_\mathbb{H}^2\d s)^{\frac{1}{2}}(\int_0^t\|u(s)-v(s)\|_{\mathbb{H}}^2\d s)^{\frac{1}{2}}\\\leq&C\left((\int_{-h}^t\|u(s)\|_\mathbb{H}^2\d s)^{\frac{1}{2}}+(\int_{-h}^t\|v(s)\|_{\mathbb{H}}^2\d s)^{\frac{1}{2}}\right)(\int_{-h}^t\|u(s)-v(s)\|_{\mathbb{H}}^2\d s)^{\frac{1}{2}}, \endaligned$$
where in the last inequality we used $\sup_t r_i'(t)<1$.
Similarly, for $u,v\in L^2_{\rm{loc}}([-h,\infty);\mathbb{H})$ and $t>0$
$$\aligned&\int_0^t\|\int_{-h}^0\left(\mathcal{P}[u(s+r)\cdot\nabla u(s)]-\mathcal{P}[v(s+r)\cdot\nabla v(s)]\right)\d r\|_{\mathbb{X}^*}\d s\\\leq&\int_0^t\int_{-h}^0\|\mathcal{P}[u(s+r)\cdot\nabla u(s)]-\mathcal{P}[v(s+r)\cdot\nabla v(s)]\|_{\mathbb{X}^*}\d r\d s\\\leq&C\left((\int_{-h}^t\|u(s)\|_\mathbb{H}^2\d s)^{\frac{1}{2}}+(\int_{-h}^t\|v(s)\|_{\mathbb{H}}^2\d s)^{\frac{1}{2}}\right)(\int_{-h}^t\|u(s)-v(s)\|_{\mathbb{H}}^2\d s)^{\frac{1}{2}}. \endaligned$$
Let $u_n$ converge to $u$ in $L^2_{\rm{loc}}([-h,\infty),\mathbb{H})$. By Lebesgue's dominated convergence theorem  we have
$$\aligned&\int_0^t\|\mathcal{P}[\int_{-h}^0f_1(r,\cdot,u_n(s+r))\d r]
-\mathcal{P}[\int_{-h}^0f_1(r,\cdot,u(s+r))\d r]
\|_{\mathbb{X}^*}\d s\\\leq&C\int_0^t\int_{-h}^0\|f_1(r,\cdot,u_n(s+r))
-f_1(r,\cdot,u(s+r))\|_{L^1(O)}\d r
\d s\rightarrow0, n\rightarrow\infty.
\endaligned$$
Similarly, the convergence for $f$  can be obtained.
Then the above estimates imply that (\ref{cd2}) in (H1) holds in this case. The convergence for $g$ can be obtained similarly.

For (H2),
since $${ }_{\mathbb{X}^*}\langle \mathcal{P}[y\cdot\nabla x],x\rangle_{\mathbb{X}}=\langle y\cdot\nabla x,x\rangle_\mathbb{H}=0, x,y\in\mathbb{X},$$
we have that  for $u\in C([-h,\infty);\mathbb{X})$,
$$\aligned{ }_{\mathbb{X}^*}\langle A_1(u(t))+A_2(u_t),u(t)\rangle_{\mathbb{X}}=&-\mathcal{N}_1(u(t))
+\int_{-h}^0\langle f_1(r,\cdot,u(t+r)),u(t)\rangle_\mathbb{H} \d r\\&+\langle f(\cdot,u(t-r_2(t))),u(t)\rangle_\mathbb{H}\\\leq&-\mathcal{N}_1(u(t))+C(1+\|u_t\|_\mathfrak{C}^2) .\endaligned$$
Now (H2) follows  by taking integration w.r.t. $t$.

By Theorem \ref{T2} the results follow.\qed

\vskip.10in
\noindent\textbf{Stochastic 2-D Navier-Stokes equation}

Now we come to the case $d=2$, $\alpha=1$.  Define
$$H_0=\mathbb{H},\quad  V= W_{0,\sigma}^{1,2}(O)
$$

 It is standard that  using the Gelfand triple
$$     V\subseteq H_0\equiv H_0^*\subseteq V^*,   $$
we see that the following mappings
$$ A_1: V\rightarrow V^*, \  A_2: \mathbb{R}^+\times (L^2_V\cap \mathfrak{C})\rightarrow V^*  $$
are well defined. For $ u \in \mathbb{S}_1$, we easily deduce that
$A_1(\cdot, u (\cdot)), A_2(\cdot, u _\cdot)\in L^1_{\rm{loc}}([0,\infty);V^*), B(\cdot, u _\cdot)\in L^2_{\rm{loc}}([0,\infty);L_2(l_2;H_0))$.

Suppose that  $f_1, f, g$ satisfy
\begin{equation}|f_1(r,\xi,u_1)-f_1(r,\xi,u_2)|+|f(\xi,u_1)-f(\xi,u_2)|\leq C(1+|u_1|^\gamma+|u_2|^\gamma)|u_1-u_2|, 1\leq \gamma\leq 2\end{equation}
\begin{equation}\|g(\xi,u_1)-g(\xi,u_2)\|_{l^2}\leq C|u_1-u_2|.\end{equation}

\begin{thm} Fix $d=2, \alpha=1$. Suppose that $f_1, f$ and $g$ satisfy (5.4), (5.5) respectively. Then for every $\psi\in L^2(-h,0;V)\cap  C([-h,0];  H_0)$,
    $(\ref{NS})$
    has a unique (probabilistically) strong solution $X\in C([-h,\infty),H_0)$ satisfying for every $T>0$
and $q\in\mathbb{N}$
$$\E^P\left(\sup_{t\in[-h,T]}\|X(t)\|_\mathbb{H}^{2q}+[\int_0^T\mathcal{N}_1(X(t))\d t]^q \right)   < \infty.$$

\end{thm}

\begin{proof}
Now we check (H4): we have for $u,v\in L^2_{\rm{loc}}([-h,\infty);V)\cap C([-h,\infty);\mathbb{H})$
$$\aligned&|\int_0^T{ }_{V^*}\langle u(s-r_1(s))\cdot \nabla u(s)-v(s-r_1(s))\cdot \nabla v(s), u(s)-v(s)\rangle_{V}\d s|\\\leq&C\int_0^T\|\nabla u(s)\|_\mathbb{H}\|u(s-r_1(s))-v(s-r_1(s))\|_{L^4}\|u(s)-v(s)\|_{L^4}
\d s\\\leq&C\int_0^T\|\nabla u(s)\|_\mathbb{H}\|u(s-r_1(s))-v(s-r_1(s))\|_{V}^{1/2}\|u(s-r_1(s))-v(s-r_1(s))\|_{\mathbb{H}}^{1/2}
\\&\|u(s)-v(s)\|_{V}^{1/2}\|u(s)-v(s)\|_{\mathbb{H}}^{1/2}
\d s\\\leq&C\int_0^T\|\nabla u(s)\|_\mathbb{H}^2\|u_s-v_s\|_{\mathfrak{C}}^2\d s+\varepsilon[\int_0^T\|u(s)-v(s)\|_{V}^2\d s]^{1/2}
[\int_0^T\|u(s-r_1(s))-v(s-r_1(s))\|_{V}^2\d s]^{1/2}
\\\leq&C\int_0^T\|\nabla u(s)\|_\mathbb{H}^2\|u_s-v_s\|_{\mathfrak{C}}^2\d s+\varepsilon\int_{-h}^T\|\nabla u(s)-\nabla v(s)\|_{\mathbb{H}}^2\d s,\endaligned$$
where we used $div v=0$ in the first inequality, Lemma 5.1 and the interpolation inequality in the second inequality, Young's inequality in the third inequality.

For $u(s)\cdot \nabla u(s)$ and $\int_{-h}^0u(s+r)\cdot \nabla u(s)\d r$ we have similar results.

Now for $f$ we have for $u,v\in L^2_{\rm{loc}}([-h,\infty);V)\cap C([-h,\infty),\mathbb{H})$
$$\aligned&|\int_0^T{ }_{V^*}\langle f(u(s-r_2(s)))-f(v(s-r_2(s))), u(s)-v(s)\rangle_{V}\d s|\\\leq&C\int_0^T\int[|u(s-r_2(s))|^{\gamma}+|v(s-r_2(s))|^{\gamma}+1]|u(s-r_2(s))-v(s-r_2(s))||u(s)-v(s)|\d \xi\d s
\\\leq&C\int_0^T[1+\|u(s-r_2(s))\|_{L^{\gamma p_1}}^{\gamma}+\|v(s-r_2(s))\|_{L^{\gamma p_1}}^{\gamma}]
\|u(s-r_2(s))-v(s-r_2(s))\|_{L^{2q_1}}\\&\|u(s)-v(s)\|_{L^{2q_1}}\d s
\\\leq&C\int_0^T[1+\|u(s-r_2(s))\|_{L^{\gamma p_1}}^{\gamma }+\|v(s-r_2(s))\|_{L^{\gamma p_1}}^{\gamma }]\|u(s-r_2(s))-v(s-r_2(s))\|_{V}^{\sigma}
\|u(s)-v(s)\|_{V}^{\sigma}\\&
\|u(s-r_2(s))-v(s-r_2(s))\|_{\mathbb{H}}^{1-\sigma}\|u(s)-v(s)\|_{\mathbb{H}}^{1-\sigma}\d s
\\\leq&C\int_0^T[1+\|u(s-r_2(s))\|_{V}^{\sigma_0\gamma \frac{1}{1-\sigma}}
\|u(s-r_2(s))\|_{\mathbb{H}}^{(1-\sigma_0)\gamma \frac{1}{1-\sigma}}\\&+\|v(s-r_2(s))\|_{V}^{\sigma_0\gamma\frac{1}{1-\sigma}}
\|v(s-r_2(s))\|_{\mathbb{H}}^{(1-\sigma_0)\gamma \frac{1}{1-\sigma}}]
\|u_s-v_s\|_{\mathfrak{C}}^{2}\d s\\&+\varepsilon\int_{-h}^T\|u(s)-v(s)\|_{V}^{2}\d s,\endaligned$$
where $\frac{1}{p_1}+\frac{1}{q_1}=1, \sigma=1-\frac{1}{q_1}<1, \sigma_0=1-\frac{2}{\gamma p_1}<1, \frac{\sigma_0\gamma}{1-\sigma}\leq 2$ and we used (5.4) in the first inequality, Lemma 5.1 and the interpolation inequality in the third inequality.
The term $\int_{-h}^0f_1(r,u(t+r))\d r$ can be treated similarly.
Since $g$ satisfies Lipschitz conditions, we easily deduce that
$$\|B(t,u_t)-B(t,v_t)\|_{L_2(l_2;H_0)}^2\leq C\|u_t-v_t\|_\mathfrak{C}^2.$$
Now (H4) follows.

Now for (H5), we have for $u\in  L^2_{\rm{loc}}([-h,\infty);V)\cap L^\infty_{\rm{loc}}([-h,\infty);\mathbb{H})\cap C([-h,\infty);\mathbb{X}^*)$
$$\aligned&\int_0^T\|u(s-r_1(s))\cdot \nabla u(s)\|_{V^*}^2\d s\leq \int_0^T\|u(s-r_1(s))\|_{L^4}^2\|u(s)\|_{L^4}^2\d s\\\leq&C\int_0^T\|u(s-r_1(s))\|_{V}\|u(s-r_1(s))\|_{\mathbb{H}}\|u(s)\|_{V}\|u(s)\|_\mathbb{H}\d s\\\leq&C[\int_0^T\|u(s-r_1(s))\|_{V}^2\|u(s-r_1(s))\|^2_{\mathbb{H}}\d s]^{1/2}[\int_0^T\|u(s)\|_{V}^2\|u(s)\|_\mathbb{H}^2\d s]^{1/2}\\\leq&C\int_{-h}^T\|u(s)\|_{V}^2\|u(s)\|_\mathbb{H}^2\d s\\\leq&C[\int_{-h}^T\|u(s)\|_{V}^2\d s]^2+C\sup_{s\in[-h,T]}\|u(s)\|_\mathbb{H}^4.\endaligned$$
The other terms can be checked easily since $f, f_1, g$ satisfy linear growth conditions.
Therefore,  existence and uniqueness of solutions
follow from Theorem 4.2.
\end{proof}

\begin{rem}The existence results in Section 2 can also be applied to other stochastic equations with delays from fluid dynamics such as the stochastic quasi-geostrophic equation with delays, which  has been studied in [RZZ14] without delays. However, to prove the uniqueness of the solutions, we need $L^p$-estimates for the solutions, which cannot be obtained by Theorem 4.2. This will be studied in the future. \end{rem}

\subsection{Stochastic semilinear equation}

Consider the following stochastic semilinear equation in a smooth domain $O\subset \mathbb{R}^d$:
\begin{equation}\label{Sem}\aligned dX(t)=&\Delta [X(t)+a_1(X(t-r_0(t)))+\int_{-h}^0a_2(r,X(t+r))\d r]\d t\\&+(f_1(X(t))+f_2(X(t-r_1(t)))+\int_{-h}^0f_3(r,X(t+r))\d r)\\&\cdot[\nabla X(t)+\nabla X(t-r_2(t))+\int_{-h}^0\nabla X(t+r)\d r]\d t\\
&+[\partial_i b_1^i(X(t))+\partial_i b_2^i(X(t-r_3(t)))+\int_{-h}^0\partial_i b_3^i(r,X(t+r))\d r]\d t\\
&+[g_1(X(t))+g_2(X(t-r_4(t)))+\int_{-h}^0g_3(r,X(t+r))\d r]\d t\\&+c_i(X(t-r_5(t)))dW^i(t),\endaligned\end{equation}
$$X(t,\xi)=0, (t,\xi)\in [-h,\infty)\times \partial O,$$
$$X(t)=\psi(t), -h\leq t\leq0,$$
where
$r_i, i=0,1,...,5,$  are continuous differentiable  functions on $\mathbb{R}$ with $0\leq r_i(t)<h, i=0,1,...,5,  \sup r_i'(t)<1, i=0,1,2,3,4$,  $a_1, f_i, b_i,  g_i, i=1,2, $ and $c$ are functions from $O\times \mathbb{R}$ to $\mathbb{R}$, $\mathbb{R}^d$, $\mathbb{R}^d$, $\mathbb{R}$ and $\mathbb{R}\times l^2$, respectively,  continuous with respect to the second variable, $a_2, f_3, b_3,  g_3$ are functions from $[-h,0]\times O\times \mathbb{R}$ to $\mathbb{R}$, $\mathbb{R}^d$, $\mathbb{R}^d$, $\mathbb{R}$, continuous with respect to the third variable,  satisfying
\begin{equation}  |g_1(\xi,u)|\leq  C[|u|^{\gamma_1}+1],\quad g_1(\xi, u)u\leq C|u|^2, (\xi,u)\in O\times\mathbb{R}, \end{equation}
for $2\leq \gamma_1<2+\frac{4}{d}$, and for some $\kappa_0>0$ and $q\in L^2(O)$
\begin{equation}\aligned&\|b_3(r,\xi,u)\|_{\mathbb{R}^d}+\|b_1(\xi,u)\|_{\mathbb{R}^d}+\|b_2(\xi,u)\|_{\mathbb{R}^d}+\|c(\xi,u)\|_{l^2}\\&+|g_2(\xi,u)|+|g_3(r,\xi,u)|\\\leq& \kappa_0 \cdot|u|+q(\xi), \quad \forall (r,\xi,u)\in [-h,0]\times O\times \mathbb{R},\endaligned\end{equation}
\begin{equation}\sup_{s}\frac{1}{1-r_0'(s)}|\partial_{u}a_1(\xi,u)|+h|\partial_{u}a_2(r,\xi,u)|\leq \frac{1}{4},\quad \forall (r,\xi,u)\in [-h,0]\times O\times \mathbb{R},\end{equation}
\begin{equation}\|\partial_{\xi}a_1(\xi,u)\|_{\mathbb{R}^{d}}+\|\partial_{\xi}a_2(r,\xi,u)\|_{\mathbb{R}^{d}}\leq C(1+|u|),\quad \forall (r,\xi,u)\in [-h,0]\times O\times \mathbb{R}.\end{equation}
Moreover, $f_i, i=1,2,$ and  $f_3$ are bounded and Lipschitz continuous with respect to the second variable and the third variable respectively.

When $f_i, a_i, b_i=0, i=1,2, 3$,  then the above equation is the stochastic reaction-diffusion
equation which has also attracted a lot of attention (see e.g. [DPZ92], [RZZ14a] and the references therein).

Below we choose
$$\mathbb{Y}=\mathbb{H}=L^2(O), \quad V=W^{1,2}_{0}(O),$$
and
$$\mathbb{X}^*=(W^{2+d,2}_{0}(O))^*,\quad \mathbb{X}=W^{2+d,2}_{0}(O). $$
 We define the operators $A$ and $B$ as follows: for $x\in C_{0}^\infty(O)$
$$A_1(x):=\Delta x+\partial_i b_1^i(x)+g_1(x),$$
and for $u\in C([-h,0];C_{0}^\infty(O))$
$$\aligned A_2(t, u):=&\Delta[a_1(u(-r_0(t)))+\int_{-h}^0a_2(r,u(r))\d r]+[f_1(u(0))+ f_2(u(-r_1(t)))+\int_{-h}^0 f_3(r,u(r))\d r]\\&\cdot[\nabla u(0)+\nabla u(-r_2(t))+\int_{-h}^0\nabla u(r)\d r]\\&+\partial_i b_2^i(u(-r_3(t)))+\int_{-h}^0\partial_i b_3^i(r,u(r))\d r+g_2(u(-r_4(t)))+\int_{-h}^0g_3(r,u(r))\d r.\endaligned$$
$$B(t,u):=c(u(-r_5(t))).$$
Let $e_n,n\in\mathbb{N}$, $\Pi_n, \mathbb{H}_n$ be as in Section 2. We define the functional $\mathcal{N}_1$ on $\mathbb{H}$ as follows:
$$\mathcal{N}_1(u)=\left\{\begin{array}{ll}\|\nabla u\|_{L^2(O)}^2&\ \ \ \ \textrm{ if } u\in W^{1,2}_{0,\sigma}(O),\\0&\ \ \ \ \textrm{ otherwise} ,\end{array}\right.$$
It is obvious that $\mathcal{N}_1\in \mathfrak{U}^2$. In this case $\mathcal{N}(\psi)<\infty$ is equivalent to $\psi\in L^2(-h,0;V)$.

Then we have for $x,y\in C_{0}^\infty(O)$, $y_n\rightarrow y$ in $\mathbb{H}$,
\begin{equation}\|\Delta x-\Delta y\|_{\mathbb{X}^*}\leq C\|x-y\|_\mathbb{H},\end{equation}
$$\aligned|{  }_{\mathbb{X}^*}\langle f_1(x)\cdot\nabla x-f_1(y)\cdot\nabla y, \varphi\rangle_{\mathbb{X}}|\leq& C|{  }_{\mathbb{X}^*}\langle f_1(x)\cdot(\nabla x-\nabla y)+(f_1(x)-f_1(y))\cdot\nabla y, \varphi\rangle_{\mathbb{X}}\\\leq&C\|x-y\|_\mathbb{H}[\|\nabla x\|_\mathbb{H}+\|\nabla y\|_\mathbb{H}+1] \|\varphi\|_{\mathbb{X}},\endaligned$$
\begin{equation}\aligned|{  }_{\mathbb{X}^*}\langle \partial_i b_1^i(y_n)-\partial_i b_1^i(y), \varphi\rangle_{\mathbb{X}}|\leq& C\|\varphi\|_{C^1(O)}\int |b_1(\xi,y_n(\xi))-b_1(\xi,y(\xi))|\d \xi\rightarrow0.\endaligned\end{equation}
Moreover, if $y_n\rightarrow y$ in $\mathbb{H}$ and weakly in $V$, we obtain
$$\aligned|{  }_{\mathbb{X}^*}\langle g_1(y_n)-g_1(y), \varphi\rangle_{\mathbb{X}}|\leq& C\|\varphi\|_{C(O)}\int |g_1(\xi,y_n(\xi))-g_1(\xi,y(\xi))|\d \xi\rightarrow0,\endaligned$$
where we used (5.7) to deduce that
$$\|g_1(\cdot,y_n(\cdot))\|_{L^{1+\varepsilon}(O)}^{1+\varepsilon}\leq C[1+\|y_n\|_{L^{(1+\varepsilon)\gamma_1}}^{(1+\varepsilon)\gamma_1}]\leq C[1+ \|y_n\|_V^{d(\frac{(1+\varepsilon)\gamma_1}{2}-1)}\|y_n\|_{\mathbb{H}}^{(1+\varepsilon)\gamma_1-d(\frac{(1+\varepsilon)\gamma_1}{2}-1)}],$$
where $d(\frac{(1+\varepsilon)\gamma_1}{2}-1)<2$.
By (5.7)-(5.10), we can extend the operators $A_1$ and $A_2, B$ to $V$ and $\mathfrak{C}\cap L^2_V$ such that
for $x\in V, u\in \mathfrak{C}\cap L^2_V$, $A_1(x)\in \mathbb{X}^*$ and $A_2(t,u)\in \mathbb{X}^*$, $B(t,u)\in L_2(l_2,\mathbb{H})$.
Thus, we can write the system in the following abstract form
$$du(t)=[A_1(u(t))+A_2(t,u_t)]\d t+B(t,u_t)\d W(t), \quad u(t)=\psi(t), -h\leq t\leq0.$$
 We can easily deduce that
$\|\Pi_nA_1(\cdot)\|_{\mathbb{X}^*}$ is locally bounded on $\mathbb{H}_n$ and
$\|\Pi_nA_2(\cdot,\cdot)\|_{\mathbb{X}^*}, \|\Pi_nB(\cdot,\cdot)\|_{L_2(l_2;\mathbb{H})}$ are  bounded on balls in $\mathbb{R}^+\times C([-h,0];\mathbb{H}_n)$. Furthermore, we easily deduce that for every $t>0$, $\Pi_n A_1(t,\cdot)$ is continuous on $\mathbb{H}_n$, $\Pi_n A_2(t,\cdot), \Pi_n B(t,\cdot)$ are continuous on $C([-h,0];\mathbb{H}_n)$.

\beg{thm}\label{Sem1}
 For every $\psi\in   \mathfrak{C}\cap L^2_V$
    $(\ref{Sem})$
    has a weak solution $X:[-h,\infty)\times \Omega\rightarrow \mathbb{H}$ such that for every $q\in\mathbb{N}$, $T>0$
$$\E^P\left(\sup_{t\in[-h,T]}\|X(t)\|_\mathbb{H}^{2q}+[\int_0^T\mathcal{N}_1(X(t))\d t]^q \right)   < \infty.$$

\end{thm}

\proof  Now we check (H1):
Let $u_n$ converge to $u$ in $L^2_{\rm{loc}}([-h,\infty);\mathbb{H})$ and weakly in $L^{2}_{\rm{loc}}([-h,\infty);V)$. We prove the convergence required in (H1) for each term. By (5.9) we have for $\varphi\in\mathbb{X}$,
$$\aligned&|\int_0^t{ }_{\mathbb{X}^*}\langle \Delta[a_1(u_n(s-r_0(s)))-a_1(u(s-r_0(s)))], \varphi\rangle_{\mathbb{X}}\d s|\\\leq&C[\int_0^t\|u_n(s-r_0(s))-u(s-r_0(s))\|_\mathbb{H}^2\d s]^\frac{1}{2}
\\\rightarrow&0,\quad n\rightarrow\infty.
\endaligned$$
The convergence for $a_2$ can be obtained similarly. Since $f_2$ is bounded and Lipschitz continuous with respect to the second variable, we have
$$\aligned&|\int_0^t{ }_{\mathbb{X}^*}\langle f_2(u_n(s-r_1(s)))\cdot\nabla u_n(s)-f_2(u(s-r_1(s)))\cdot\nabla u(s), \varphi\rangle_{\mathbb{X}}\d s|\\\leq&C[\int_0^t\|\nabla u_n(s)\|_\mathbb{H}^2\d s]^\frac{1}{2}[\int_0^t\|u_n(s-r_1(s))-u(s-r_1(s))\|_\mathbb{H}^2\d s]^\frac{1}{2}
\\&+\int_0^t{ }_{\mathbb{X}^*}\langle f_2(u(s-r_1(s)))[ \nabla u_n(s)-\nabla u(s)], \varphi\rangle_{\mathbb{X}}\d s\\\leq&C[\int_0^t\|\nabla u_n(s)\|_\mathbb{H}^2\d s]^\frac{1}{2}[\int_{-h}^t\|u_n(s)-u(s)\|_\mathbb{H}^2\d s]^\frac{1}{2}
\\&+\int_0^t{ }_{\mathbb{X}^*}\langle f_2(u(s-r_1(s)))[\nabla u_n(s)-\nabla u(s)], \varphi\rangle_{\mathbb{X}}\d s\\\rightarrow&0,\quad n\rightarrow\infty,
\endaligned$$
where we used that $u_n$ converge to $u$  weakly in $L^{2}_{\rm{loc}}([-h,\infty);V)$ to deduce $\int_0^T \|u_n(s)\|_V^2\d s\leq M$ in the last convergence.
The convergence for $f_1(u(s))\cdot\nabla u(s),\int_{-h}^0f_3(r,u(s+r))\d r\cdot\nabla u(s)$ can be obtained similarly. Moreover, we have
$$\aligned&|\int_0^t{ }_{\mathbb{X}^*}\langle f_2(u_n(s-r_1(s)))\cdot\nabla u_n(s-r_2(s))-f_2(u(s-r_1(s)))\cdot\nabla u(s-r_2(s)), \varphi\rangle_{\mathbb{X}}\d s|\\\leq&C[\int_{-h}^t\|\nabla u_n(s)\|_\mathbb{H}^2\d s]^\frac{1}{2}[\int_0^t\|u_n(s-r_1(s))-u(s-r_1(s))\|_\mathbb{H}^2\d s]^\frac{1}{2}
\\&+\int_0^t{ }_{\mathbb{X}^*}\langle f_2(u(s-r_1(s)))\cdot[\nabla u_n(s-r_2(s))-\nabla u(s-r_2(s))], \varphi\rangle_{\mathbb{X}}\d s\\\rightarrow&0,\quad n\rightarrow\infty,
\endaligned$$
where we used that $u_n$ converge to $u$  weakly in $L^{2}_{\rm{loc}}([-h,\infty);V)$ in the last convergence. Furthermore, we get that
$$\aligned&|\int_0^t\int_{-h}^0{ }_{\mathbb{X}^*}\langle f_2(u_n(s-r_1(s)))\cdot\nabla u_n(s+r)-f_2(u(s-r_1(s)))\cdot\nabla u(s+r), \varphi\rangle_{\mathbb{X}}\d r\d s|\\\leq&Ch[\int_{-h}^t\|\nabla u_n(s)\|_\mathbb{H}^2\d s]^\frac{1}{2}[\int_0^t\|u_n(s-r_1(s))-u(s-r_1(s))\|_\mathbb{H}^2\d s]^\frac{1}{2}
\\&+\int_0^t\int_{-h}^0{ }_{\mathbb{X}^*}\langle f_2(u(s-r_1(s)))[\nabla u_n(s+r)-\nabla u(s+r)], \varphi\rangle_{\mathbb{X}}\d r\d s\\\leq&Ch[\int_{-h}^t\|\nabla u_n(s)\|_\mathbb{H}^2\d s]^\frac{1}{2}[\int_0^t\|u_n(s-r_1(s))-u(s-r_1(s))\|_\mathbb{H}^2\d s]^\frac{1}{2}
\\&+\int_{-h}^0\int_0^t{ }_{\mathbb{X}^*}\langle f_2(u(s-r_1(s)))[\nabla u_n(s+r)-\nabla u(s+r)], \varphi\rangle_{\mathbb{X}}\d s\d r\\\rightarrow&0,\quad n\rightarrow\infty,
\endaligned$$
where we used $u_n$ converge to $u$  weakly in $L^{2}_{\rm{loc}}([-h,\infty);V)$  and the Lebesgue's dominated convergence theorem in the last convergence. For the other terms in $[f_1(u(t))+ f_2(u(t-r_1(t)))+\int_{-h}^0 f_3(r,u(t+r))\d r]\cdot[\nabla u(t-r_2(t))+\int_{-h}^0\nabla u(t+r)\d r]$ we have similar estimates and obtain similar convergence required in (H1). Now for $b_2$, by (5.8) we have
$$\aligned&|\int_0^t{ }_{\mathbb{X}^*}\langle \partial_i b_2^i(u_n(s-r_3(s)))-\partial_i b_2^i(u(s-r_3(s))), \varphi\rangle_{\mathbb{X}}\d s|\\\leq&C\|\varphi\|_{C^1(O)}\int_0^t\|b_2(u_n(s-r_3(s)))-b_2(u(s-r_3(s)))\|_{L^1(O)}\d s
\\\rightarrow&0,\quad n\rightarrow\infty,
\endaligned$$
where the last convergence follows from the Lebesgue's dominated convergence theorem. The convergence for $b_1, b_3$ can be obtained similarly.
For  cut-off function $\chi_R$ as in (H1), we also have
$$\aligned&\int_0^t{ }_{\mathbb{X}^*}\int |g_1(u_n(s))\chi_R(\|u_n(s)\|_\mathbb{H})-g_1(u(s))\chi_R(\|u(s)\|_\mathbb{H})|^{1+\varepsilon} |\varphi|^{1+\varepsilon}\d \xi\d s\\\leq&C\|\varphi\|_{C(O)}^{1+\varepsilon}\int_0^t\|g_1(u_n(s))\chi_R(\|u_n(s)\|_\mathbb{H})-g_1(u(s))\chi_R(\|u(s)\|_\mathbb{H})\|_{L^{1+\varepsilon}(O)}^{1+\varepsilon}\d s
\\\leq& C\|\varphi\|_{C(O)}^{1+\varepsilon}\int_0^t[1+\|u_n(s)\|_{L^{\gamma_1(1+\varepsilon)}}^{\gamma_1(1+\varepsilon)}
\chi_R(\|u_n(s)\|_\mathbb{H})+\|u(s)\|_{L^{\gamma_1{(1+\varepsilon)}}}^{\gamma_1(1+\varepsilon)}\chi_R(\|u(s)\|_\mathbb{H})]\d s\\\leq& C\int_0^t[1+ \|u_n(s)\|_V^{d(\frac{\gamma_1(1+\varepsilon)}{2}-1)}\chi_R(\|u_n(s)\|_\mathbb{H})\|u_n(s)\|_\mathbb{H}^{\gamma_1(1+\varepsilon)-d(\frac{\gamma_1(1+\varepsilon)}{2}-1)}
\\&+\|u(s)\|_V^{d(\frac{\gamma_1(1+\varepsilon)}{2}-1)}\chi_R(\|u(s)\|_\mathbb{H})\|u(s)\|_\mathbb{H}^{\gamma_1
(1+\varepsilon)-d(\frac{\gamma_1(1+\varepsilon)}{2}-1)}]\d s\\\leq&C\int_0^t[1+ \|u_n(s)\|_V^{d(\frac{\gamma_1(1+\varepsilon)}{2}-1)}+\|u(s)\|_V^{d(\frac{\gamma_1(1+\varepsilon)}{2}-1)}]\d s.
\endaligned$$
Since $d(\frac{\gamma_1(1+\varepsilon)}{2}-1)\leq 2$ we obtain that
$$\int_0^t{ }_{\mathbb{X}^*}\langle g_1(u_n(s))\chi_R(\|u_n(s)\|_\mathbb{H})-g_1(u(s))\chi_R(\|u(s)\|_\mathbb{H}), \varphi\rangle_{\mathbb{X}}\d s\rightarrow0, n\rightarrow\infty.$$
For $g_2$ we have
$$\aligned&\int_0^t{ }_{\mathbb{X}^*}\langle g_2(u_n(s-r_4(s)))-g_2(u(s-r_4(s))), \varphi\rangle_{\mathbb{X}}\d s\\\leq&C\|\varphi\|_{C(O)}\int_0^t\|g_2(u_n(s-r_4(s)))-g_2(u(s-r_4(s)))\|_{L^1(O)}\d s
\\\rightarrow&0,n\rightarrow\infty,
\endaligned$$
where the last convergence follows from the Lebesgue's dominated convergence theorem.
The convergence for $\int_{-h}^0g_3(r,u(t+r))\d r$ can be obtained similarly. By Remark 2.2 (i), (5.11) and (5.12), we obtain
$$\int_0^t{ }_{\mathbb{X}^*}\langle \Delta u_n(s)\chi_R(\|u_n(s)\|_\mathbb{H})-\Delta u(s)\chi_R(\|u(s)\|_\mathbb{H}), \varphi\rangle_{\mathbb{X}}\d s\rightarrow0, n\rightarrow\infty.$$
$$\int_0^t{ }_{\mathbb{X}^*}\langle \partial_i b_1^i (u_n(s))\chi_R(\|u_n(s)\|_\mathbb{H})-\partial_i b_1^i(u(s))\chi_R(\|u(s)\|_\mathbb{H}), \varphi\rangle_{\mathbb{X}}\d s\rightarrow0, n\rightarrow\infty.$$
 Then combining  the estimates above, (2.3) and (2.4) in (H1) follow. By (5.8) we obtain that
if $u_n$ converge to $u$ in $L^2_{\rm{loc}}([-h,\infty);\mathbb{H})$, then
$$\lim_{n\rightarrow\infty}\int_0^t\|B^*(s,u_s^n)(v)-B^*(s,u_s)(v)\|_{l_2}\d s=0.$$
 Now (H1) follows.

For (H2); by (5.7)-(5.10)
we have for $u\in C([-h,\infty),\mathbb{X})$ and $T>0$
$$\aligned&\int_0^T{ }_{\mathbb{X}^*}\langle A_1(u(t))+A_2(u_t),u(t)\rangle_{\mathbb{X}}\d t\\=&\int_0^T[-\mathcal{N}_1(u(t))+{ }_{\mathbb{X}^*}\langle\Delta [a_1(u(t-r_0(t)))+\int_{-h}^0a_2(u(t+r))\d r],u(t)\rangle_{\mathbb{X}}
+\langle(f_1(u(t))+f_2(u(t-r_1(t)))\\&+\int_{-h}^0f_3(r,u(t+r))\d r)\cdot(\nabla u(t)+\nabla u(t-r_2(t))+\int_{-h}^0\nabla u(t+r)\d r), u(t)\rangle\\
&-\langle b_1^i(u(t))+b_2^i(u(t-r_3(t)))+\int_{-h}^0b_3^i(r,u(t+r))\d r, \partial_i u(t)\rangle\\
&+\langle g_1(u(t))+g_2(u(t-r_4(t)))+\int_{-h}^0g_3(r,u(t+r))\d r, u(t)\rangle]\d t\\\leq&\int_0^T[-\frac{7}{8}\mathcal{N}_1(u(t))+C(1+\|u_t\|_\mathfrak{C}^2)]\d t+\int_{-h}^0\mathcal{N}_1(u(t))\d t\\&+\int_0^T\int  [[|\partial_u a_1(u(t-r_0(t)))||\nabla u(t-r_0(t))|+|\partial_\xi a_1(u(t-r_0(t)))|]|\nabla u(t)|\\&+\int_{-h}^0[|\partial_u a_2(u(t+r))||\nabla u(t+r)|+|\partial_\xi a_2(u(t+r))|]|\nabla u(t)|\d r]\d \xi\d t\\\leq&\int_0^T[-\frac{3}{4}\mathcal{N}_1(u(t))+C(1+\|u_t\|_\mathfrak{C}^2)]\d t+\int_{-h}^0\mathcal{N}_1(u(t))\d t+\left(\int_0^T\int|\nabla u(t)|^2\d\xi\d t\right)^{1/2}\\&[\left(\int_0^T\int  [|\partial_u a_1(u(t-r_0(t)))|^2|\nabla u(t-r_0(t))|^2]\d\xi\d t \right)^{1/2}\\&+h^{1/2}\left(\int_0^T\int \int_{-h}^0|\partial_u a_2(u(t+r))|^2|\nabla u(t+r)|^2\d r\d \xi\d t\right)^{1/2}]\\\leq&\int_0^T[-\frac{1}{2}\mathcal{N}_1(u(t))+C(1+\|u_t\|_\mathfrak{C}^2)]\d t+2\int_{-h}^0\mathcal{N}_1(u(t))\d t .\endaligned$$

For (H3); Similar as the computation above we have  $u\in L^2_{\rm{loc}}([-h,\infty);V)\cap L_{\rm{loc}}^\infty([-h,\infty);\mathbb{H})\cap C([-h,\infty);\mathbb{X}^*)$ and every $t>0$
$$\|A_1(u(t))\|_{\mathbb{X}^*}\leq C(1+\|u_t\|_\mathfrak{C}+\|u(t)\|_V^{d(\frac{\gamma_1}{2}-1)}\|u(t)\|_\mathbb{H}^{(\gamma_1-d(\frac{\gamma_1}{2}-1))}),$$
$$\aligned\|A_2(t,u_t)\|_{\mathbb{X}^*}\leq& C[1+\|u(t-r_0(t))\|_\mathbb{H}+\|u_t\|_\mathfrak{C}+\|\nabla u(t)\|_\mathbb{H}\\&+\|\nabla u(t-r_2(t))\|_\mathbb{H}+\int_{-h}^0\|\nabla u(t+r)\|_\mathbb{H}\d r]
,\endaligned$$
$$\|B(t,u_t)\|^2_{L_2(l^2;\mathbb{H})}\leq  C(1+\|u_t\|_\mathfrak{C}^2),$$
where $C$ is a constant changing from line to line.
Now for $1<\gamma<2$ with $d\gamma(\frac{\gamma_1}{2}-1)<2$ we have for every $T>0$
$$\aligned& \int_0^T[\|A_1(u(t))\|_{\mathbb{X}^*}^\gamma+\|A_2(t,u_t)\|_{\mathbb{X}^*}^\gamma]\d t\\\leq&\int_0^T C(1+ \|u(t)\|_V^{d\gamma(\frac{\gamma_1}{2}-1)}\|u(t)\|_\mathbb{H}^{\gamma(\gamma_1-d(\frac{\gamma_1}{2}-1))}+\|u_t\|_\mathfrak{C}^2+\|u(t)\|_V^2)\d t+\int_{-h}^0 C\|u(t)\|_V^2\d t,\endaligned$$
$$\|B(t,u_t)\|^2_{L_2(l^2;\mathbb{H})}\leq  C(1+\|u_t\|_\mathfrak{C}^2).$$
Now (H3) follows. Then the results follow from Theorem 2.2.\qed

Now we assume that
\begin{equation}\aligned&\sum_{i=1}^2|g_i(\xi,u_1)-g_i(\xi, u_2)|+|g_3(r,\xi,u_1)-g_3(r,\xi, u_2)|\\\leq& C[|u_1|^{\gamma_3}+|u_2|^{\gamma_3}+1]|u_1-u_2|,\forall (r,\xi, u_j)\in[-h,0]\times O\times \mathbb{R}, j=1,2, \endaligned \end{equation}
for $1\leq \gamma_3\leq\frac{4}{d}$.
$$\|c(\xi,u_1)-c(\xi,u_2)\|_{l^2}\leq C|u_1-u_2|.$$
$\partial_u a_1, \partial_u a_2$ do not depend on $u$.

If $d\leq 2$, suppose that
\begin{equation}\aligned&\sum_{i=1}^2\|b_i(\xi,u_1)-b_i(\xi, u_2)\|_{\mathbb{R}^d}+\|b_3(r,\xi,u_1)-b_3(r,\xi, u_2)\|_{\mathbb{R}^d}\\\leq& C(1+|u_1|+|u_2|)|u_1-u_2|,\forall (r,\xi, u_j)\in[-h,0]\times O\times \mathbb{R}, j=1,2,\endaligned \end{equation}
\begin{equation}\aligned&\|\partial_\xi a_1(\xi, u_1)-\partial_\xi a_1(\xi, u_2)\|_{\mathbb{R}^d}+\|\partial_\xi a_2(r,\xi, u_1)-\partial_\xi a_2(r,\xi, u_2)\|_{\mathbb{R}^d}\\\leq& C(1+|u_1|+|u_2|)|u_1-u_2|,\forall (r,\xi, u_j)\in[-h,0]\times O\times \mathbb{R}, j=1,2,\endaligned \end{equation}
If $d\geq 3$, $f_1, f_2, f_3$ do not depend on $u$, $b_1, b_2, \partial_\xi a_1$ and $b_3, \partial_\xi a_2$ are Lipschitz continuous with respect to the second variable and the third variable, respectively.

Choose $H_0=\mathbb{H}$.

If $\gamma_1\leq (1+\frac{4}{d})\wedge2(1+\frac{1}{d})-\varepsilon$ for some $\varepsilon>0$,
we have that for $u\in L^2_V\cap \mathfrak{C}$, $$\|g_1(u)\|_{V^*}\leq \|g_1(u)\|_{L^q}\leq C(1+\|u\|_V^{d(\frac{\gamma_1}{2}-\frac{1}{q})}\|u\|_\mathbb{H}^{\gamma_1-d(\frac{\gamma_1}{2}-\frac{1}{q})}),$$
where $q=\frac{2d}{d+2}\vee (1+\varepsilon_0)$ for some $\varepsilon_0$ and we used Lemma 5.1 in the first inequality.
From this we see that the following mappings
$$ A_1: V\rightarrow V^*, \  A_2: \mathbb{R}^+\times (L^2_V\cap \mathfrak{C})\rightarrow V^*  $$
are well defined. For $ u \in \mathbb{S}_1$, we easily deduce that
$A_1(\cdot, u (\cdot)), A_2(\cdot, u _\cdot)\in L^1_{\rm{loc}}([0,\infty);V^*), B(\cdot, u _\cdot)\in L^2_{\rm{loc}}([0,\infty);L_2(l_2;H_0))$.

\beg{thm}\label{Tsm} Suppose that all the assumptions above are satisfied with $\gamma_1\leq (1+\frac{4}{d})\wedge2(1+\frac{1}{d})-\varepsilon$ for some $\varepsilon>0$.
 For every $\psi\in L^2(-h,0;V)\cap  C([-h,0];  H_0) ,$
    $(\ref{Sem})$
    has a unique (probabilistically) strong solution $X\in C([-h,\infty);H_0)$ satisfying for every $T>0, q\in\mathbb{N}$
$$\E^P\left(\sup_{t\in[-h,T]}\|X(t)\|_\mathbb{H}^{2q}+[\int_0^T\mathcal{N}_1(X(t))\d t]^q \right)   < \infty.$$

\proof  Now we check (H4):

When $d\geq 3$ for $ f_1, f_2, f_3$  it is obvious that (H4) holds. Now we only check the case  when $d\leq 2$: for $u,v\in L^2_{\rm{loc}}([-h,\infty);V)\cap C([-h,\infty);\mathbb{H})$, $T>0$, we have
$$\aligned&\int_0^T{ }_{V^*}\langle f_2(u(s-r_1(s)))\cdot\nabla u(s-r_2(s))-f_2(v(s-r_1(s)))\cdot\nabla v(s-r_2(s)), u(s)-v(s)\rangle_{V}\d s\\\leq&C\int_0^T\|\nabla u(s-r_2(s))\|_\mathbb{H}\|u(s-r_1(s))-v(s-r_1(s))\|_{L^4}\|u(s)-v(s)\|_{L^4}
\d s\\&+\varepsilon\int_0^T\|\nabla u(s-r_2(s))-\nabla v(s-r_2(s))\|_\mathbb{H}^2\d s+C\int_0^T\|u(s)-v(s)\|_\mathbb{H}^2\d s\\\leq&C\int_0^T\|\nabla u(s-r_2(s))\|_\mathbb{H}\|u(s-r_1(s))-v(s-r_1(s))\|_{V}^{1/2}\|u(s-r_1(s))-v(s-r_1(s))\|_{\mathbb{H}}^{1/2}\\&\|u(s)-v(s)\|_{V}^{1/2}\|u(s)-v(s)\|_{\mathbb{H}}^{1/2}
\d s+\varepsilon\int_{-h}^T\|\nabla u(s)-\nabla v(s)\|_\mathbb{H}^2\d s+C\int_0^T\|u(s)-v(s)\|_\mathbb{H}^2\d s
\\\leq&C\int_0^T\|\nabla u(s-r_2(s))\|_\mathbb{H}^2\|u_s-v_s\|_{\mathfrak{C}}^2\d s+\varepsilon[\int_0^T\|u(s)-v(s)\|_{V}^2\d s]^{1/2}
\\&[\int_0^T\|u(s-r_1(s))-v(s-r_1(s))\|_{V}^2\d s]^{1/2}+\varepsilon\int_{-h}^T\|\nabla u(s)-\nabla v(s)\|_\mathbb{H}^2\d s+C\int_0^T\|u(s)-v(s)\|_\mathbb{H}^2\d s
\\\leq&C\int_0^T\|\nabla u(s-r_2(s))\|_\mathbb{H}^2\|u_s-v_s\|_{\mathfrak{C}}^2+\varepsilon\int_{-h}^T\|\nabla u(s)-\nabla v(s)\|_\mathbb{H}^2\d s+C\int_0^T\|u(s)-v(s)\|_\mathbb{H}^2\d s.\endaligned$$
We also have for $u,v\in L^2_{\rm{loc}}([-h,\infty);V)\cap C([-h,\infty);\mathbb{H})$, $T>0$,
$$\aligned&\int_0^T{ }_{V^*}\langle f_2(u(s-r_1(s)))\cdot\int_{-h}^0\nabla u(s+r)\d r-f_2(v(s-r_1(s)))\cdot\int_{-h}^0\nabla v(s+r)\d r, u(s)-v(s)\rangle_{V}\d s\\\leq&C\int_0^T\int_{-h}^0\|\nabla u(s+r)\|_\mathbb{H}\d r\|u(s-r_1(s))-v(s-r_1(s))\|_{L^4}\|u(s)-v(s)\|_{L^4}
\d s\\&+\varepsilon\int_0^T\int_{-h}^0\|\nabla u(s+r)-\nabla v(s+r)\|_\mathbb{H}^2\d r\d s+C\int_0^T\|u(s)-v(s)\|_\mathbb{H}^2\d s
\\\leq&C\int_0^T\int_{-h}^0\|\nabla u(s+r)\|_\mathbb{H}^2\d r\|u_s-v_s\|_{\mathfrak{C}}^2\d s+\varepsilon[\int_0^T\|u(s)-v(s)\|_{V}^2\d s]^{1/2}
\\&[\int_0^T\|u(s-r_1(s))-v(s-r_1(s))\|_{V}^2\d s]^{1/2}+\varepsilon h\int_{-h}^T\|\nabla u(s)-\nabla v(s)\|_\mathbb{H}^2\d s+C\int_0^T\|u(s)-v(s)\|_\mathbb{H}^2\d s
\\\leq&C\int_0^T\int_{-h}^0\|\nabla u(s+r)\|_\mathbb{H}^2\d r\|u_s-v_s\|_{\mathfrak{C}}^2\d s+\varepsilon \int_{-h}^T\|\nabla u(s)-\nabla v(s)\|_\mathbb{H}^2\d s+C\int_0^T\|u(s)-v(s)\|_\mathbb{H}^2\d s.\endaligned$$
We obtain similar bounds for the other terms including $f_1, f_3$.
Now when $d\geq 3$ we have for $u,v\in L^2_{\rm{loc}}([-h,\infty);V)\cap C([-h,\infty);\mathbb{H})$, $s>0$
$$\aligned&|{ }_{V^*}\langle b_2^i(u(s-r_3(s)))-b_2^i(v(s-r_3(s))), \partial_i u(s)-\partial_i v(s)\rangle_{V}|\\\leq&C\|u_s-v_s\|_\mathfrak{C}^2
+\varepsilon\|\nabla u(s)-\nabla v(s)\|_\mathbb{H}^2.
\endaligned$$
We obtain similar bounds for $ b_3$.

When $d\leq 2$  we
 have  for $u,v\in L^2_{\rm{loc}}([-h,\infty);V)\cap C([-h,\infty);\mathbb{H})$, $T>0$,
$$\aligned&\int_0^T|{ }_{V^*}\langle b_2^i(u(s-r_3(s)))-b_2^i(v(s-r_3(s))), \partial_i u(s)-\partial_i v(s)\rangle_{V}|\d s\\\leq&C\int_0^T\|b_2(u(s-r_3(s)))-b_2(v(s-r_3(s)))\|_\mathbb{H}^2\d s
+\varepsilon\int_0^T\|\nabla u(s)-\nabla v(s)\|_\mathbb{H}^2\d s\\\leq&C\int_0^T[1+\|u(s-r_3(s))\|_{L^4}^2+\|v(s-r_3(s))\|_{L^4}^2]\|u(s-r_3(s))-v(s-r_3(s))\|_{L^4}^2\d s
\\&+\varepsilon\int_0^T\|\nabla u(s)-\nabla v(s)\|_\mathbb{H}^2\d s\\\leq&C\int_{0}^T[1+\|u(s-r_3(s))\|_{V}\|u(s-r_3(s))\|_{\mathbb{H}}+\|v(s-r_3(s))\|_{V}\|v(s-r_3(s))\|_{\mathbb{H}}]\\&\|u(s-r_3(s))-v(s-r_3(s))
\|_{V}\|u(s-r_3(s))-v(s-r_3(s))\|_{\mathbb{H}}\d s+\varepsilon\int_0^T\|\nabla u(s)-\nabla v(s)\|_\mathbb{H}^2\d s\\\leq&C\int_{0}^T[1+\|u(s-r_3(s))\|_{V}^2\|u(s-r_3(s))\|_{\mathbb{H}}^2+\|v(s-r_3(s))\|_{V}^2\|v(s-r_3(s))\|^2_{\mathbb{H}}]\|u_s-v_s\|_{\mathfrak{C}}^2\d s
\\&+\varepsilon\int_{-h}^T\|u(s)-v(s)\|_V^2\d s.
\endaligned$$
The other terms including $b_1, b_3$ can be estimated similarly.
For $a_1$ we
 have that for $u,v\in L^2_{\rm{loc}}([-h,\infty);V)\cap C([-h,\infty);\mathbb{H})$, $T>0$,
$$\aligned&\int_0^T|{ }_{V^*}\langle\Delta( a_1(u(s-r_0(s)))-a_1(v(s-r_0(s)))),  u(s)- v(s)\rangle_{V}|\d s\\\leq&C\int_0^T\|\partial_\xi a_1(u(s-r_0(s)))-\partial_\xi a_1(v(s-r_0(s)))\|_\mathbb{H}^2\d s
+(\varepsilon+\frac{1}{4})\int_{-h}^T\|\nabla u(s)-\nabla v(s)\|_\mathbb{H}^2\d s.
\endaligned$$
Then by the above estimates for $b_2$, $a_1$ can be estimated similarly. We also have similar estimates for $a_2$.

For $g$  we
 have for $u,v\in L^2_{\rm{loc}}([-h,\infty);V)\cap C([-h,\infty);\mathbb{H})$, $s>0$
$$\aligned&{ }_{V^*}\langle g_1(u(s))-g_1(v(s)), u(s)-v(s)\rangle_{V}\\\leq&C\int[|u(s)|^{\gamma_3}+|v(s)|^{\gamma_3}+1]|u(s)-v(s)|^{2}\d \xi
\\\leq&C[1+\|u\|_{L^{\gamma_3p}}^{\gamma_3}+\|v\|_{L^{\gamma_3p}}^{\gamma_3}]\|u-v\|_{L^{2q}}^2
\\\leq&C[1+\|u\|_{L^{\gamma_3p}}^{\gamma_3}+\|v\|_{L^{\gamma_3p}}^{\gamma_3}]\|u-v\|_{V}^{2\sigma}\|u-v\|_{\mathbb{H}}^{2(1-\sigma)}
\\\leq&C[1+\|u\|_{L^{\gamma_3p}}^{\gamma_3\frac{1}{1-\sigma}}+\|v\|_{L^{\gamma_3p}}^{\gamma_3\frac{1}{1-\sigma}}]
\|u-v\|_{\mathbb{H}}^{2}+\varepsilon\|u-v\|_{V}^{2}\\\leq&C[1+\|u\|_{V}^{\sigma_0\gamma_3\frac{1}{1-\sigma}}
\|u\|_{\mathbb{H}}^{(1-\sigma_0)\gamma_3\frac{1}{1-\sigma}}+\|v\|_{V}^{\sigma_0\gamma_3\frac{1}{1-\sigma}}\|v\|_{\mathbb{H}}^{(1-\sigma_0)\gamma_3\frac{1}{1-\sigma}}]
\|u-v\|_{\mathbb{H}}^{2}+\varepsilon\|u-v\|_{V}^{2},\endaligned$$
where $\frac{1}{p}+\frac{1}{q}=1, \sigma=(\frac{1}{2}-\frac{1}{2q})d<1, \sigma_0=(\frac{1}{2}-\frac{1}{\gamma_3p})d<1, \frac{\sigma_0\gamma_3}{1-\sigma}\leq 2$.
Similarly we have for $u,v\in L^2_{\rm{loc}}([-h,\infty);V)\cap C([-h,\infty);\mathbb{H})$, $T>0$
$$\aligned&\int_0^T{ }_{V^*}\langle g_2(u(s-r_4(s)))-g_2(v(s-r_4(s))), u(s)-v(s)\rangle_{V}\\\leq&C\int_0^T\int[|u(s-r_4(s))|^{\gamma_3}+|v(s-r_4(s))|^{\gamma_3}+1]|u(s-r_4(s))-v(s-r_4(s))||u(s)-v(s)|\d \xi
\\\leq&C\int_0^T[1+\|u(s-r_4(s))\|_{L^{\gamma_3p}}^{\gamma_3}+\|v(s-r_4(s))\|_{L^{\gamma_3p}}^{\gamma_3}]\|u(s-r_4(s))-v(s-r_4(s))\|_{L^{2q}}\\&\|u(s)-v(s)\|_{L^{2q}}\d s
\\\leq&C\int_0^T[1+\|u(s-r_4(s))\|_{L^{\gamma_3p}}^{\gamma_3}+\|v(s-r_4(s))\|_{L^{\gamma_3p}}^{\gamma_3}]\|u(s-r_4(s))-v(s-r_4(s))\|_{V}^{\sigma}
\|u(s)-v(s)\|_{V}^{\sigma}\\&
\|u(s-r_4(s))-v(s-r_4(s))\|_{\mathbb{H}}^{1-\sigma}\|u(s)-v(s)\|_{\mathbb{H}}^{1-\sigma}\d s
\\\leq&C\int_0^T[1+\|u(s-r_4(s))\|_{V}^{\sigma_0\gamma_3\frac{1}{1-\sigma}}
\|u(s-r_4(s))\|_{\mathbb{H}}^{(1-\sigma_0)\gamma_3\frac{1}{1-\sigma}}\\&+\|v(s-r_4(s))\|_{V}^{\sigma_0\gamma_3\frac{1}{1-\sigma}}
\|v(s-r_4(s))\|_{\mathbb{H}}^{(1-\sigma_0)\gamma_3\frac{1}{1-\sigma}}]
\|u_s-v_s\|_{\mathfrak{C}}^{2}\d s\\&+\varepsilon\int_{-h}^T\|u(s)-v(s)\|_{V}^{2}\d s.\endaligned$$
$g_3$ can be estimated similarly.

For (H5);
for $a$ we have for $u\in L^2_{\rm{loc}}([-h,\infty); V)\cap C([-h,\infty);\mathbb{X}^*)\cap L^\infty_{\rm{loc}}(-h,\infty;\mathbb{H}), T>0$
$$\aligned&\int_0^T\|\Delta a_1(u(t-r_0(t)))\|_{V^*}^2\d t\\\leq& C\int_0^T[\|\partial_u a_1(u(t-r_0(t)))\nabla u(t-r_0(t))\|_{\mathbb{H}}^2
+\|\partial_\xi a_1(u(t-r_0(t)))\|_{\mathbb{H}}^2]\d t\\\leq& C\int_0^T[\|\nabla u(t-r_0(t))\|_{\mathbb{H}}^2
+\|u(t-r_0(t))\|_{\mathbb{H}}^2+1]\d t\\\leq&C\int_{-h}^T[\|\nabla u(t)\|_{\mathbb{H}}^2
+\|u(t)\|_{\mathbb{H}}^2+1]\d t.\endaligned$$
Since  $\gamma_1\leq (1+\frac{4}{d})\wedge 2(1+\frac{1}{d})-\varepsilon$, by Lemma 5.1 and the interpolation inequality, we have for  $u\in L^2_{\rm{loc}}([-h,\infty);V)\cap L_{\rm{loc}}^\infty([-h,\infty);\mathbb{H})\cap C([-h,\infty);\mathbb{X}^*)$ and every $T>0$
$$\aligned\|A_1(u(t))\|_{V^*}\leq &C(1+\|u_t\|_\mathfrak{C}+\|u(t)\|_V^{d(\frac{\gamma_1}{2}-\frac{1}{q})}\|u(t)\|_\mathbb{H}^{(\gamma_1-d(\frac{\gamma_1}{2}-\frac{1}{q}))}+\|\nabla u(t)\|_\mathbb{H}),\endaligned$$
$$\aligned\|A_2(t,u_t)\|_{V^*}\leq& C[1+\|u(t-r_0(t))\|_V+\|u_t\|_\mathfrak{C}+\|\nabla u(t)\|_\mathbb{H}\\&+\|\nabla u(t-r_2(t))\|_\mathbb{H}+\int_{-h}^0\|\nabla u(t+r)\|_\mathbb{H}\d r]
,\endaligned$$
$$\|B(t,u_t)\|^2_{L_2(l^2;\mathbb{H})}\leq  C(1+\|u_t\|_\mathfrak{C}^2),$$
where $q=\frac{2d}{d+2}\vee(1+\varepsilon)$, $W^{-1,2}\subset L^q$ continuously  and $C$ is a constant changing from line to line.
Now  we have for every $T>0$
$$\aligned& \int_0^T[\|A_1(u(t))\|_{V^*}^2+\|A_2(t,u_t)\|_{V^*}^2]\d t\\\leq&\int_0^T C(1+ \|u(t)\|_V^{2d(\frac{\gamma_1}{2}-\frac{1}{q})}\|u(t)\|_\mathbb{H}^{2(\gamma_1-d(\frac{\gamma_1}{2}-\frac{1}{q}))}+\|u(t)\|_V^2+\|u_t\|_\mathfrak{C}^2)\d t+\int_{-h}^0 C\|u(t)\|_V^2\d t\\&+\int_0^T(\|\nabla u(t)\|_\mathbb{H}^2+\|\nabla u(t-r_2(t))\|^2_\mathbb{H}+\int_{-h}^0\|\nabla u(t+r)\|^2_\mathbb{H}\d r)\d t\\\leq& [\int_0^T C(1+ \|u(t)\|_V^{2})\d t]^2+\sup_{t\in[-h,T]}\|u(t)\|_\mathbb{H}^{4(\gamma_1-d(\frac{\gamma_1}{2}-\frac{1}{q}))}+\int_{-h}^0 C\|u(t)\|_V^2\d t\\&+C\int_{-h}^T\|\nabla u(t)\|_\mathbb{H}^2\d t+C(1+\sup_{t\in[-h,T]}\|u(t)\|_\mathbb{H}^2),\endaligned$$
and
$$\|B(t,u_t)\|^2_{L_2(l^2,\mathbb{H})}\leq  C(1+\|u_t\|_\mathfrak{C}^2).$$
Now (H5) follows. Then the results follow from Theorem 4.2.\qed

\end{thm}

\subsection{Stochastic generalized Porous Medium Euqations}

 For $k\geq0, p>1$, the dual space of $W_0^{k,p}(O)$ is given by $W^{-k,p'}(O)$, where $p'=\frac{p}{p-1}$. The following Sobolev embeddings hold:
$$W_0^{k,p}(O)\subset C^m(\bar{O}), 0\leq m<k-\frac{d}{p}.$$ By Poincare's inequality we have for $x\in W_0^{1,2}(O)$
$$\int_O|x(\xi)|^2\d \xi\leq \rho_O\int_O|\nabla x(\xi)|^2\d \xi.$$

Consider the following quasi linear SPDE with Dirichlet boundary conditions:
\begin{equation}\label{PM}\aligned dX(t)=&[\partial_{ij}^2a^{ij}(\xi,X(t))+\partial_ib^i(\xi,X(t))+c(\xi,X(t))]\d t\\&+[\partial_ib_1^i(\xi,X(t-r_1(t)),X(t))+\int_{-h}^0\partial_ib_2^i(r,\xi,X(t+r),X(t))\d r+c_1(\xi,X(t-r_2(t)))\\&+\int_{-h}^0c_2(r,\xi,X(t+r))\d r]\d t
+\sigma_i(\xi,X(t-r_3(t)))\d W^i(t),\endaligned\end{equation}
$$X(t,\xi)=0, \quad (t,\xi)\in \mathbb{R}^+\times \partial O, $$
$$X(t)=\psi(t), -h\leq t\leq 0,$$
where  $r_i, i=1,2,3,$  are continuous differentiable  functions on $\mathbb{R}$ satisfying $-h\leq r_i\leq 0, i=1,2,3$ and $\sup r_i'(t)<1, i=1,2$, $a,b,c, c_1$ and $\sigma$ are functions from $O\times \mathbb{R}$ to $\mathbb{R}^{d^2}$, $\mathbb{R}^d$, $\mathbb{R}$, $\mathbb{R}$ and $l^2$ respectively,  continuous  with respect to the second variable; $b_1(\xi,u_1,u_2)$ is a  function from $O\times \mathbb{R}\times \mathbb{R}$ to   $\mathbb{R}^d$,  continuous with respect to $(u_1,u_2)$;  $ b_2(r,\xi,u_1,u_2)$ is a function from $[-h,0]\times O\times \mathbb{R}\times \mathbb{R}$ to  $\mathbb{R}^d$, continuous with respect to $(u_1,u_2)$;  $c_1(\xi,u_1)$ is a function from $ O\times \mathbb{R}$ to $\mathbb{R}$, $c_2(r,\xi,u_1)$ is a function from $[-h,0]\times O\times \mathbb{R}$ to $\mathbb{R}$,  continuous with respect to the third variable, and satisfy for some fixed $q\geq2$ and all $\xi\in O, u_1,u_2\in\mathbb{R}, r\in[-h,0]$:
$$\partial_ua^{ij}(\xi,u_1)x_ix_j\geq \kappa_{a,0}|u_1|^{q-2}|x|^2, \quad x\in\mathbb{R}^d,$$
$$\|a(\xi,u_1)\|_{\mathbb{R}^{d^2}}\leq \kappa_{a,1}(|u_1|^{q-1}+1),$$
$$\|\partial_ja^{\cdot j}(\xi,u_1)\|_{\mathbb{R}^d}+\|b(\xi,u_1)\|_{\mathbb{R}^d}\leq \kappa_{a,b}|u_1|^{q-1}+\kappa_{a,b}'|u_1|^{\frac{q}{2}},$$
$$|c(\xi,u_1)|\leq \kappa_{c,1}|u_1|^{q-1}+\kappa_{c,2}(|u_1|+1),$$
$$\|b_1(\xi,u_1,u_2)\|_{\mathbb{R}^d}\leq C|u_1||u_2|^{\frac{q}{2}-1},$$
$$\|b_2(r,\xi,u_1,u_2)\|_{\mathbb{R}^d}\leq C|u_1||u_2|^{\frac{q}{2}-1},$$
$$|c_1(\xi,u_1)|\leq \kappa_{c,2}(|u_1|+1),$$
$$|c_2(r,\xi,u_1)|\leq \kappa_{c,2}(|u_1|+1),$$
$$\|\sigma(\xi,u_1)\|_{l^2}\leq \kappa_\sigma (|u_1|+1),$$
where all $\kappa$ with subscripts are strictly positive constants, and
$$\frac{\kappa_{a,b}}{2}(1+\frac{q^2\rho_O}{4})+\frac{\kappa_{c,1}\cdot q^2\rho_O}{4}\leq \frac{\kappa_{a,0}}{2}.$$
We choose
$$V=\mathbb{Y}:=L^q(O),\quad \mathbb{H}:=L^2(O),$$
and
$$\mathbb{X}:=W^{d+2,2}_0(O), \quad \mathbb{X}^*:=W^{-d-2,2}(O).$$
Let $e_n,n\in\mathbb{N},$,  $\Pi_n, \mathbb{H}_n$ be as in Section 2. Define the functional $\mathcal{N}_1$ on $\mathbb{Y}$ as follows:
$$\mathcal{N}_1(y):=\left\{\begin{array}{ll}\frac{\kappa_0}{q^2}\int |\nabla (|y(\xi)|^{\frac{q}{2}-1}y(\xi))|^2\d \xi&\ \ \ \ \textrm{ if } |y|^{\frac{q}{2}-1}y\in W^{1,2}_0(O),\\+\infty&\ \ \ \textrm{  otherwise }.\end{array}\right.$$
Then by [GRZ09 Lemma 5.1] we know $\mathcal{N}_1\in\mathfrak{U}^q$. In this case, if $\psi\in L^q (-h,0; W_0^{1\vee\frac{d+\varepsilon}{2},2})$ for some $\varepsilon>0$, we have $\mathcal{N}_1(\psi)\leq C\|\psi\|_{W_0^{1\vee\frac{d+\varepsilon}{2},2}}$, then $\mathcal{N}(\psi)<\infty$.

Define for $x\in V=L^q(O)$
$$A_1(x):=\partial_{ij}^2a^{ij}(\cdot,x(\cdot))+\partial_ib^i(\cdot,x(\cdot))+c(\cdot,x(\cdot))\in\mathbb{X}^*,$$
and for $u\in L_V^q\cap \mathfrak{C}$,
$$\aligned A_2(t,u):=&\partial_ib_1^i(\cdot,u(-r_1(t)),u(0))+\int_{-h}^0\partial_ib_2^i(r,\cdot,u(r),u(0))\d r+c_1(\cdot,u(-r_2(t)))\\&+\int_{-h}^0c_2(r,\cdot,u(r))\d r\in\mathbb{X}^*,\endaligned$$
$$B(t,u):=\sigma(\cdot,u(-r_3(t)))\in L_2(l^2;\mathbb{H}).$$
Thus, we can write the system in the following abstract form
$$\d X(t)=[A_1(X(t))+A_2(t,X_t)]\d t+B(t,X_t)\d W(t), \quad u(0)=u_0.$$
 We can easily deduce that
$\|\Pi_nA_1(\cdot)\|_{\mathbb{X}^*}$ is locally bounded on $\mathbb{H}_n$ and
$\|\Pi_nA_2(\cdot,\cdot)\|_{\mathbb{X}^*}, \|\Pi_nB(\cdot,\cdot)\|_{L_2(l_2;\mathbb{H})}$ are  bounded on balls in $\mathbb{R}^+\times C([-h,0];\mathbb{H}_n)$. Furthermore, we easily deduce that for every $t>0$ $\Pi_n A_1(t,\cdot)$ is continuous on $\mathbb{H}_n$ and $\Pi_n A_2(t,\cdot), \Pi_n B(t,\cdot)$ are continuous on $C([-h,0];\mathbb{H}_n)$.

\beg{thm}\label{Tpm}
 For every $\psi\in L^q(-h,0; V)\cap  \mathfrak{C}$ with $\mathcal{N}(\psi)<\infty$
    $(\ref{PM})$
    has a weak solution $X:[-h,\infty)\times \Omega\rightarrow \mathbb{H}$ such that for every $q_1\in\mathbb{N}$, $T>0$
$$\E^P\left(\sup_{t\in[-h,T]}\|X(t)\|_\mathbb{H}^{2q_1}+[\int_0^T\mathcal{N}_1(X(t))dt]^{q_1} \right)   < \infty.$$

\end{thm}

\proof Now for (H1); let $u_n$ converge to $u$ in $L^2_{\rm{loc}}([-h,\infty);\mathbb{H})$ and weakly in $L^{q}_{\rm{loc}}([-h,\infty);V)$.  We have for  cut-off function $\chi_R$ and every $t>0$
$$\aligned&|\int_0^t{ }_{\mathbb{X}^*}\langle \partial^2_{ij}a(u_n(s)), \varphi\rangle_{\mathbb{X}}\chi_R(\|u_n(s)\|_\mathbb{H})-{ }_{\mathbb{X}^*}\langle\partial^2_{ij}a(u(s)), \varphi\rangle_{\mathbb{X}}\chi_R(\|u(s)\|_\mathbb{H})\d s|
\\=&|\int_0^t\int [a(u_n(s))\chi_R(\|u_n(s)\|_\mathbb{H})-a(u(s))\chi_R(\|u(s)\|_\mathbb{H})] \partial^2_{ij}\varphi\d\xi\d s|.
\endaligned$$
Moreover, we obtain for all $n\in\mathbb{N}$
$$\aligned &\int_0^t\int |a(u_n(s))\chi_R(\|u_n(s)\|_\mathbb{H})-a(u(s))\chi_R(\|u(s)\|_\mathbb{H})|^{\frac{q}{q-1}}\d\xi\d s\\\leq&C\|\varphi\|_{C^2(O)}\int_0^t\int [|u_n(s)|^{q}+|u(s)|^{q}+1]\d \xi\d s\leq M,\endaligned$$
which implies that  for every $t>0$
$$\aligned&\int_0^t{ }_{\mathbb{X}^*}\langle \partial^2_{ij}a(u_n(s)), \varphi\rangle_{\mathbb{X}}\chi_R(\|u_n(s)\|_\mathbb{H})-{ }_{\mathbb{X}^*}\langle\partial^2_{ij}a(u(s)), \varphi\rangle_{\mathbb{X}}\chi_R(\|u(s)\|_\mathbb{H})\d s
\rightarrow0.
\endaligned$$
Similar convergence also hold for the terms including $b$ and $c$. Then (2.3) follows. Now for (2.4) we have for every $t>0$
$$\aligned&|\int_0^t{ }_{\mathbb{X}^*}\langle \partial_i b_1^i(u_n(s-r_1(s)),u_n(s))-\partial_i b_1^i(u(s-r_1(s)),u(s)), \varphi\rangle_{\mathbb{X}}\d s|\\\leq&C\|\varphi\|_{C^1(O)}\int_0^t\int|b_1(u_n(s-r_1(s)),u_n(s))-b_1(u(s-r_1(s)),u(s))|\d\xi\d s
\\\leq&C\|\varphi\|_{C^1(O)}\int_0^t\int [|u_n(s-r_1(s))||u_n(s)|^{q/2-1}+|u(s-r_1(s))||u(s)|^{q/2-1}]\d \xi\d s.
\endaligned$$
Now for $1<\gamma<\frac{q}{q-1}$ we have for every $t>0$
$$\aligned&\int_0^t\int [|u_n(s-r_1(s))|^\gamma|u_n(s)|^{(q/2-1)\gamma}+|u(s-r_1(s))|^\gamma|u(s)|^{(q/2-1)\gamma}]\d \xi\d s
\\\leq&\int_0^t\int |u_n(s-r_1(s))|^2+|u(s-r_1(s))|^2+C[|u_n(s)|^{q}+|u(s)|^{q}+1]\d\xi \d s\\\leq &C\int_{-h}^t[\|u_n(s)\|_{\mathbb{H}}^2+\|u(s)\|_{\mathbb{H}}^2+\|u_n(s)\|_V^q+\|u(s)\|_V^q+1]\d s.\endaligned$$
Then we obtain for every $t>0$
$$\int_0^t{ }_{\mathbb{X}^*}\langle \partial_i b_1^i(u_n(s-r_1(s)),u_n(s))-\partial_i b_1^i(u(s-r_1(s)),u(s)), \varphi\rangle_{\mathbb{X}}\d s\rightarrow0,n\rightarrow\infty.$$
The convergence for $b_2$ can be obtained similarly. Now for $c_1$ we have for every $t>0$
$$\aligned&\int_0^t{ }_{\mathbb{X}^*}\langle c_1(u_n(s-r_2(s)))-c_1(u(s-r_2(s))), \varphi\rangle_{\mathbb{X}}\d s
\\=&\int_0^t\int (c_1(u_n(s-r_2(s)))-c_1(u(s-r_2(s))))\varphi\d\xi\d s\endaligned$$
Similarly, we obtain that
$$\aligned& \int_0^t\int |c_1(u_n(s-r_2(s)))-c_1(u(s-r_2(s)))|\d\xi\d s\\\leq&C\|\varphi\|_{C(O)}\int_0^t\int [|u_n(s-r_2(s))|+|u(s-r_2(s))|+1]\d \xi\d s.
\endaligned$$
Similarly as above we obtain for every $t>0$
$$\int_0^t{ }_{\mathbb{X}^*}\langle c_1(u_n(s-r_2(s)))-c_1(u(s-r_2(s))), \varphi\rangle_{\mathbb{X}}\d s\rightarrow0, n\rightarrow\infty.$$
We obtain the similar convergence results for $c_2$. Then combining  the estimates above,  (2.4) in (H1) follows. By the assumptions for $\sigma$ we obtain that
if $u_n$ converge to $u$ in $L^2_{\rm{loc}}([-h,\infty);\mathbb{H})$ then
$$\lim_{n\rightarrow\infty}\int_0^t\|B^*(s,u_s^n)(v)-B^*(s,u_s)(v)\|_{l_2}\d s=0.$$
So (H1) follows.

Now for (H2);
by a similar calculation in [GRZ09, Section 5] we obtain that
$$\aligned&\int_0^T{ }_{\mathbb{X}^*}\langle A_1(u(t))+A_2(u_t),u(t)\rangle_{\mathbb{X}}\d t\\\leq&\int_0^T[-\mathcal{N}_1(u(t))+C(1+\|u_t\|_\mathfrak{C}^2)]\d t
\\&-\int_0^T[{ }_{\mathbb{X}^*}\langle b_1^i(u(t-r_1(t)),u(t))+\int_{-h}^0b_2^i(r,u(t+r),u(t))\d r, \partial_i u(t)\rangle_{\mathbb{X}}\\&+{ }_{\mathbb{X}^*}\langle c_1(u(t-r_2(t)))+\int_{-h}^0c_2(r,u(t+r))\d r, u(t)\rangle_{\mathbb{X}}]\d t\\\leq&\int_0^T[-\mathcal{N}_1(u(t))+C(1+\|u_t\|_\mathfrak{C}^2)
\\&+C\int \left(|u(t-r_1(t))||u(t)|^{\frac{q}{2}-1}|\nabla u(t)|+\int_{-h}^0|u(t+r)||u(t)|^{\frac{q}{2}-1}|\nabla u(t)|\d r\right)\d \xi\\&+{ }_{\mathbb{X}^*}\langle c_1(u(t-r_2(t)))+\int_{-h}^0c_2(r,u(t+r))\d r, u(t)\rangle_{\mathbb{X}}]\d t\\\leq&\int_0^T[-\frac{1}{2}\mathcal{N}_1(u(t))+C(1+\|u_t\|_\mathfrak{C}^2)]\d t+c\int_{-h}^0\mathcal{N}_1(u(t))\d t .\endaligned$$

Now for (H3); for $u\in C([-h,\infty);\mathbb{X})$ we have
$$\aligned&\|\partial^2_{ij}a(u(s))\|_{\mathbb{X}^*}
\leq C\int [|u(s)|^{q-1}+1]\d \xi\leq C(\mathcal{N}_1(u(s))+1).
\endaligned$$
Now for $b$ we have
$$\aligned&\|\partial_i b_1^i(u(s-r_1(s)),u(s))\|_{\mathbb{X}^*}\leq C\|b_1(u(s-r_1(s)),u(s))\|_{L^1(O)}
\\\leq&C\int |u(s-r_1(s))||u(s)|^{q/2-1}\d \xi.
\endaligned$$
Now for $1<\gamma<\frac{q}{q-1}$ we have
$$\aligned&\int |u(s-r_1(s))|^\gamma|u(s)|^{(q/2-1)\gamma}\d \xi
\\\leq&\int |u(s-r_1(s))|^2+C[|u(s)|^{q}+1]\d\xi \\\leq &C[\|u_s\|_{\mathfrak{C}}^2+\mathcal{N}_1(u(s))+1].\endaligned$$
Similarly, we obtain that
$$\| \int_{-h}^0\partial_i b_2^i(u(s+r),u(s))\|_{\mathbb{X}^*}^\gamma\leq C[\|u_s\|_{\mathfrak{C}}^2+\mathcal{N}_1(u(s))+1].$$
 Now for $c_1, c_2$ we also have
$$\aligned&\|c_1(u(s-r_2(s)))\|_{\mathbb{X}^*}+\|\int_{-h}^0c_2(r,u(r))\d r\|_{\mathbb{X}^*}
\\\leq& C[\|u_s\|_{\mathfrak{C}}+1].
\endaligned$$
$$\|B(t,u_t)\|^2_{L_2(l^2;\mathbb{H})}\leq  C(1+\|u_t\|_\mathfrak{C}^2).$$
Then (H3) follows. Now the results follow from Theorem 2.2.\qed

\vskip.10in

Now we prove the uniqueness. We consider the case where
\begin{equation}\aligned dX(t)=&[\Delta\Phi(X(t))+c(\xi)X(t)+c_1(\xi)X(t-r_2(t))\\&+\int_{-h}^0c_2(r,\xi)X(t+r)\d r]\d t
+\sigma_i(\xi)X(t-r_3(t))\d W^i(t),\endaligned\end{equation}
$$X(t,\xi)=0, \quad (t,\xi)\in \mathbb{R}^+\times \partial O, $$
$$X(t)=\psi(t), -h\leq t\leq 0,$$
where $c, c_1, c_2, \|\sigma\|_{l^2}\in W^{1,2}_0(O)$ and $\Phi:\mathbb{R}\rightarrow\mathbb{R}$ is a continuous function satisfying $(\Phi(s)-\Phi(t))(s-t)\geq0$, $\Psi(s)s\geq m_1|s|^q-m_2$, $|\Psi(s)|\leq m_3|s|^{q-1}+m_4$, $s,t\in\mathbb{R}, m_1,m_3,m_4>0,m_2\geq0$.
In this case, we choose the following Gelfand triple:
$$V=L^q(O)\subset W^{-1,2}(O)=:H_0\cong H_0^*=: W_0^{1,2} (O)\subset (L^q(O))^*=V^*.$$
From [PR07, Chapter 4], we only need to check
(H4) for $c$:
we have
$$\aligned&{ }_{V^*}\langle c_1(\xi)u(t-r_2)-c_1(\xi)v(t-r_2), u(t)-v(t)\rangle_V\\\leq&C\|c_1(\xi)(u(t-r_2)-v(t-r_2))\|_{H_0}\|(u(t)-v(t))\|_{H_0}\\\leq&C\|u_t-v_t\|_{C([-h,0];H_0)}^2.\endaligned$$
For $c, c_2, \sigma$ we have similar estimates.
Then we obtain:

\beg{thm}\label{Tsm}
 For every $\psi\in L^q(-h,0;V)\cap  C([-h,0];  H_0) $ with $\mathcal{N}(\psi)<\infty$,
    (5.17)
    has a unique (probabilistically) strong solution $X\in C([-h,\infty),H_0)$ satisfying for every $T>0, q_1\in\mathbb{N}$
$$\E^P\left(\sup_{t\in[-h,T]}\|X(t)\|_\mathbb{H}^{2q_1}+[\int_0^T\mathcal{N}_1(X(t))\d t]^{q_1} \right)   < \infty.$$

\end{thm}

\vskip.10in


\end{document}